\documentclass[11pt, twoside, reqno]{amsart}

\usepackage[utf8]{inputenc}
\usepackage{lmodern}

\usepackage{hyperref} 
\usepackage{xcolor}    
\usepackage{esint}
\usepackage{accents}

\definecolor{darkcyan}{cmyk}{1, 0, 0, 0.6}
\hypersetup{
    colorlinks=true,
    linkcolor=darkcyan,  
    urlcolor=blue,    
    citecolor=darkcyan 
}

\usepackage{amssymb}
    \usepackage{amsmath}
\usepackage{enumerate}
\usepackage{tikz}
\usepackage{comment}

\usepackage{mathrsfs}
\usepackage{stmaryrd}
\usepackage{enumitem}
\usepackage{dirtytalk}
\usepackage{tikz-cd} 
\newtheorem{thm}{Theorem}[section]
\newtheorem{prop}[thm]{Proposition}
\newtheorem{coro}[thm]{Corollary}
\newtheorem{lemma}[thm]{Lemma}
\theoremstyle{definition}

\newtheorem{defi}[thm]{Definition}
\theoremstyle{remark}
\newtheorem{remark}[thm]{Remark}
\numberwithin{equation}{section}
\usepackage[margin=1.2in]{geometry}
\newcommand*\dif{\mathop{}\!\mathrm{d}}

\newcommand{\eps}{\epsilon}

\newcommand{\ubar}[1]{\underaccent{\bar}{#1}}

\newcommand{\R}{\mathbb R}

\newcommand{\mc}[1]{\mathcal{#1}}

\newcommand{\mr}[1]{\mathrm{#1}}
\newcommand{\bs}[1]{\boldsymbol{#1}}
\newcommand{\br}[1]{\bs{\mathrm{#1}}}
\newcommand{\ms}[1]{\mathsf{#1}}
\newcommand{\msc}[1]{\mathscr{#1}}

\newcommand{\wt}[1]{\widetilde{#1}}

\DeclareMathOperator*{\esssup}{ess\,sup}
\DeclareMathOperator*{\essinf}{ess\,inf}

\usepackage{orcidlink}

\title[ $2\times 2$ Systems of Conservation Laws]{Liouville type theorem and Kinetic Formulation for $2\times 2$ Systems of Conservation Laws}
\author{Fabio Ancona}
\address{Dipartimento di Matematica \say{Tullio Levi Civita}, Università di Padova}
\email{ancona@math.unipd.it  \textnormal{(Fabio Ancona)}}
\author{Elio Marconi}
\address{Dipartimento di Matematica \say{Tullio Levi Civita}, Università di Padova}
\email{elio.marconi@unipd.it  \textnormal{(Elio Marconi)}}
\author{Luca Talamini}
\address{Mathematics Area, SISSA, Trieste}
\email{ltalamin@sissa.it  \textnormal{(Luca Talamini)}}
\begin{document}

\begin{abstract}
    We study $\mathbf L^\infty$ entropy solutions to $2\times 2$ systems of conservation laws. We show that, if a uniformly convex entropy exists, these solutions satisfy a pair of  kinetic equations (nonlocal in velocity), which are then shown to characterize all solutions with finite entropy production.
    Next, we prove a Liouville-type theorem for genuinely nonlinear systems, which is the main result of the paper. This implies in particular that for every finite entropy solution, every point $(t,x) \in \mathbb R^+\times \mathbb R\setminus \br J$ is of vanishing mean oscillation, where $\br J \subset \mathbb R^+\times \mathbb R$ is a set of Hausdorff dimension at most 1.
\end{abstract}

\maketitle

\section{Introduction}
We consider $2\times 2$ hyperbolic systems of conservation laws in one space dimension
\begin{equation}\label{eq:systemi}
\partial_t\,\bs  u(t,x) + \partial_x \, f(\bs  u(t,x)) = 0, \qquad \text{in $\mathscr D^\prime_{t,x}$} \qquad \bs u \in \mc U
\end{equation}
where $\mc U \subset \mathbb R^2$ is a bounded open set. We assume that the system is hyperbolic, that is, $Df$ is diagonalizable with real eigenvalues $\lambda_1, \lambda_2$ that satisfy
\begin{equation}\label{eq:hypeig}
\lambda_1(\bs u) < \lambda_2(\bs u) \qquad \forall \; \bs u \in \mc U.
\end{equation}
It is well known that in the setting of nonlinear conservation laws additional conditions must be imposed on distributional solutions in order to select the physically relevant ones: entropy solutions are weak solutions to \eqref{eq:systemi} that in addition satisfy the entropy inequality 
\begin{equation}
    \partial_\eta(\bs u) + \partial_x q(\bs u) \leq 0 \qquad \text{in $\mathscr D^\prime_{t,x}$}
\end{equation}
for every entropy-entropy flux pair $(\eta(\bs u), q(\bs u)) \in \mathbb R \times \mathbb R$ such that 
$$
\nabla q(\bs u) = \nabla \eta(\bs u) D f(\bs u), \qquad \text{$\eta$ convex.}
$$

The existence of entropy solutions is commonly investigated using relaxation techniques, by approximation schemes (such as front tracking or Glimm scheme), or by approximating the equation adding smoothing viscosity terms. Consider for example the viscous approximations with identity viscosity matrix: it is well known that if the viscous approximations $\bs u^\varepsilon$, solving 
\begin{equation}\label{eq:sysvisc}
\partial_t \bs u^\varepsilon + \partial_x f(\bs u^\varepsilon) = \varepsilon \bs u^\varepsilon_{xx}, \qquad \bs u^\varepsilon: \mathbb R^+\times \mathbb R \to \mc U
\end{equation}
converge in $\mathbf L^1_{loc}$ to a function $\bs u$, then $\bs u$ is an entropy solution to \eqref{eq:systemi}. We refer to \cite{Bre00}, \cite{Daf16} for a general introduction to the subject.

The compactness in the strong topology of the family $\{\bs u^\varepsilon\}_{\varepsilon}$ is a delicate subject. Under the existence of a bounded domain $\mc U$ for \eqref{eq:sysvisc} where \eqref{eq:hypeig} is satisfied, the method of compensated compactness developed by Tartar \cite{Tar79}, first adapted by DiPerna \cite{DiP83a} to handle the case of nonlinear hyperbolic conservation laws, allows to prove the strong compactness of the family $\{\bs u^\varepsilon\}_{\varepsilon}$, under standard nonlinearity assumptions on the flux $f$, known as \textit{genuine nonlinearity} (see Definition \ref{defi:GNL}). For a more recent account on this topic we refer to \cite[Chapter 9]{Ser00}, \cite{Daf16}. We remark that a general result on the boundedness in $\mathbf L^\infty$ of the sequence $\{\bs u^\varepsilon\}_{\varepsilon}$ is lacking, and the existence of such domain $\mc U$ must be checked each time (see e.g. \cite{DiP83b}, \cite{LPS96} where the problem is solved for classical systems of gas dynamics).

 Since the method of compensated compactness is not constructive, the structure and regularity of $\mathbf L^\infty$ solutions obtained in this way is at the moment completely unknown, apart for few very special exceptions, which are systems of Temple class \cite{AC05}, and the system of isentropic gas dynamics with $\gamma = 3$, to which various authors dedicated some attention due to its very particular and simple structure, and proved regularity in terms of traces and fractional Sobolev spaces \cite{LPT94b}, \cite{Gol23}, \cite{Vas99}. See also \cite{Tal24}, or the forthcoming paper \cite{AMT25}, for improvements upon the available fractional regularity, and for a proof of the concentration of the entropy dissipation measures on a $1$ dimensional rectifiable set.

In this paper, we first study $\mathbf L^\infty$ entropy solutions to $2\times 2$ systems of conservation laws. We show that, if a uniformly convex entropy exists, entropy solutions are in particular finite entropy solutions, and we show that the latter are characterized by a pair of kinetic equations nonlocal in the kinetic variable (Theorem \ref{thm:kin}). Next, we prove a Liouville-type theorem for genuinely nonlinear systems (Theorem \ref{thm:Liuv}), which is the main result of the paper, stating that global isentropic solutions must be constant. This implies in particular that, for every finite entropy solution, there exists a candidate jump set  $\br J \subset \mathbb R^+\times \mathbb R$ of Hausdorff dimension at most 1 such that every point $(t,x) \in \mathbb R^+\times \mathbb R\setminus \br J$ is of vanishing mean oscillation.

\subsection{Related literature}
The well posedeness theory of hyperbolic systems of conservation laws in one space dimension is rather complete for initial data with \textit{small $BV$ norm}, for which one can obtain a priori $BV$ bounds on the vanishing viscosity approximations \cite{BB05} with viscosity given by the identity matrix, for general hyperbolic $n\times n$ systems. As proved recently in \cite{BDL23}, such solutions are unique in the setting of small $BV$ solutions which satisfy the Liu admissibility condition. When restricting to special classes of genuinely nonlinear $2\times 2$ systems, more general uniqueness results are available \cite{CKV22}.
For initial data with small oscillation (i.e. close in $\mathbf L^\infty$ to a constant) the famous result by Glimm and Lax \cite{GL70} shows that there exist solutions whose $BV$ norm decays in time. These solutions are conjectured to be unique in some \textit{intermediate spaces}, see \cite{ABB23}, \cite{ABM25}, but this remains an open problem. In the same small-oscillation setting of the Glimm-Lax theorem a recent and notable result  \cite{Gla24} shows that solutions obtained with the front-tracking method propagate fractional-$BV$ regularity. Finally, in \cite{CVY24} it is proved that \textit{continuous} (possibly non entropic) solutions are not unique, differently from the scalar (multi-$d$) case \cite{BBM17}, \cite{Sil18}.

 In the setting of $\mathbf L^\infty$ solutions to $1d$ systems of conservation laws, with no smallness assumption on the initial datum, in analogy with the scalar multi-d conservation law \cite{DOW03, DLR03}, it is expected that, even if entropy solutions are not generally $BV$ starting from a general $\mathbf L^\infty$ initial datum, solutions should be $BV$-like. By this we mean that, at least if the flux is genuinely nonlinear,  there is a $1$-rectifiable set $\br J \subset \mathbb R^+\times \mathbb R$ such that 
\begin{enumerate}
    \item for any convex entropy-entropy flux $\eta, q$  the dissipation measure
    \begin{equation}\label{eq:entropydiss}
    \mu_{\eta} := \partial_t \eta(\bs u) + \partial_x q(\bs u)
    \end{equation}
    is concentrated on $\br J$; 
    \item Every point $(t,x) \in \br J^c$ is a Lebesgue point of $\bs u$
\end{enumerate}
or even better 
\begin{enumerate}
    \item[($2^\prime$)] Every point $(t,x) \in \br J^c$ is a continuity point of $\bs u$.
\end{enumerate}
We remark that (1) is known to be true in the case of solutions with finite entropy production to scalar conservation laws in one dimension \cite{BM17}, and in the same paper it is proved that ($2^\prime$) is true when the flux does not contain affine components. Property (1) is proved more generally for finite entropy solutions to scalar conservation laws in $1d$: in \cite{Mar22} for strictly convex fluxes, and in \cite{Tal24}, \cite{AMT25} for general weakly nonlinear fluxes.
In general space dimension $d > 1$ (1), (2) are still open for general fluxes for which $\{f^\prime(v)\; | \;  v \in I\}$ is not contained in an hyperplane for every interval $I$, but for partial results in the scalar case see \cite{Mar19}, \cite{Sil18}. 

Recent examples presented in \cite{BCZ18} suggest that, for genuinely nonlinear  $2\times 2$ systems of conservation laws, the total variation of a solution can potentially become infinite in finite time, even when starting from a BV initial datum. This behavior contrasts with the scalar case, where the $BV$ norm is decreasing in time thanks to Kru\v{z}kov theorem \cite{Kru77}. Therefore as mentioned in \cite{DLR03} it would be even more relevant to obtain a $BV$-like structure for solutions to $2\times 2$ system of conservation laws since $BV$ bounds are probably not available, not even for initial data with bounded variation. For other related models where $BV$ bounds are not available see \cite{AT24a, AT24b, Mar21}.

\subsection{Contributions of the present paper}

It is well known that systems of two conservation laws, differently from systems of $n$ conservation laws, $n > 2$, admit infinitely many entropies (see Definition \ref{def:eefp}). Building on this fact, Perthame \& Tzavaras \cite{PT00} constructed a family of \emph{discontinuous} entropies and derived a kinetic formulation for entropy solutions of the system of elastodynamics.
Our first contributions is to show that, for general $2\times 2$ systems admitting a uniformly convex entropy, with these entropies at hand one is able to derive a pair of kinetic equations of nonlocal type for all solutions obtained with the vanishing viscosity-compensated compactness method (see Theorem \ref{thm:kin}). In contrast with the kinetic formulations that can be obtained in the scalar case \cite{LPT94a}, \cite{DOW03}, in the context of $2\times 2$ systems the kinetic equations are \textit{nonlocal} in the kinetic variable. In the rest of the paper, we restrict to domains $\mc U$ of the form 
$$
\mc U = (\phi_1, \phi_2)(\mc W)
$$
where 
$$
\mc W = \big[\ubar w, \, \bar w \big] \times \big[{\ubar z}, \, \bar z  \big]
$$
is a rectangle in the Riemann invariant coordinates (defined by \eqref{eq:riemdef}), although there are not serious obstruction in working with more general convex domains. 

The main results of this paper are based only on a kinetic formulation, which in turn is equivalent to the following notion of finite entropy solution.
\begin{defi}
   We say that $\bs u: \Omega \to \mc U$ is a \emph{finite entropy solution} to \eqref{eq:system} if for every entropy-entropy flux pair $\eta,q  \in  C^2$ it holds
    \begin{equation}\label{eq:muetadef}
\partial_t \eta(\bs u) + \partial_x q(\bs u) = \mu_{\eta}, \qquad \text{$\mu_{\eta} \in \msc M(\Omega)$ locally finite measure}.
    \end{equation}
    We say that $\bs u$ is \emph{isentropic} if $\mu_{\eta} = 0$ for all entropy-entropy flux pairs $(\eta, q) \in C^2$. 
\end{defi}

It is a simple observation that entropy solutions are in particular finite entropy solutions, if a uniformly convex entropy exists.
\begin{prop}\label{prop:e-fe}
    Let $\bs u: \Omega \subset \mathbb R^+ \times \mathbb R \to \mc U$ be an entropy solution, and assume that there exists a uniformly convex entropy $E : \mc U \to \mathbb R$. Then $\bs u$ is also a finite entropy solution.
\end{prop}
\begin{proof}
    Let $\eta$ be any $C^2$ entropy, and let $\kappa > 0$ big enough such that $\eta + \kappa E$ is convex. Then we have
    $$
    \begin{aligned}
        \partial_t E(\bs u) + \partial_x G(\bs u) = \mu_E,\\
        \partial_t (\eta + \kappa E)(\bs u) + \partial_x (q + \kappa G)(\bs u) = \mu_{\eta + \kappa E}
    \end{aligned}
    $$
    where $G$ is the entropy flux of $E$, and $\mu_E$, $\mu_{\eta + \kappa E}$ are locally finite negative measures, because $\bs u$ is entropic. Then 
    $$
    \partial_t \eta(\bs u) + \partial_x q(\bs u) = \mu_{\eta + \kappa E} - \kappa \mu_E
    $$
    which proves the result.
\end{proof}
We observe that for $2\times 2$ systems of conservation laws, uniformly convex entropies always exist under standard assumptions, see \cite[Chapter 12]{Daf16}.

\begin{thm}\label{thm:kin}
Let $\bs u: \Omega \to \mc U$ be a finite entropy solution to \eqref{eq:system},  and define $\bs \chi_{\bs u}, \bs \psi_{\bs u}$ and $\bs \upsilon_{\bs u}, \bs \varphi_{\bs u}$ as in \eqref{eq:kfdef}, \eqref{eq:kfdef1}. Then there are locally finite measure $\mu_0,\mu_1 \in \msc M(\Omega \times ({\ubar w}, \bar w))$ and $\nu_0,\nu_1 \in \msc M(\Omega \times ({\ubar z}, \bar z))$ such that
     \begin{equation}\label{eq:kin22}
        \; \; \partial_t \bs \chi_{\bs u}(t,x,\xi) + \partial_x \bs \psi_{\bs u}(t,x,\xi) = \partial_{\xi} \mu_1 + \mu_0 \qquad \text{in $\msc D^\prime$}\big(\Omega \times ({\ubar w}, \bar w) \big)
    \end{equation}
     \begin{equation}\label{eq:kin221}
        \partial_t \bs \upsilon_{\bs u}(t,x,\zeta) + \partial_x \bs \varphi_{\bs u}(t,x,\zeta) = \partial_\zeta \nu_1 + \nu_0 \qquad \text{in $\msc D^\prime$}\big(\Omega \times ({\ubar z}, \bar z) \big)
    \end{equation}
    Moreover, $\bs u$ is an isentropic solution if and only if \eqref{eq:kin22}, \eqref{eq:kin221} hold with $\mu_i = \nu_i = 0$.
\end{thm}

Here $\bs \chi_{\bs u},  \bs \psi_{\bs u}$ are functions supported on the hypograph of the first Riemann invariant $\phi_1$ (see Section \ref{sec:singen}), and similarly $\bs \upsilon_{\bs u},\bs \varphi_{\bs u}$ are supported on the hypograph of the second Riemann invariant $\phi_2$.
The observation that allows to use the assumption about genuine nonlinearity in this kinetic setting is that when $\xi$ is close to the first Riemann invariant of the system, one has 
$$
\bs \psi_{\bs u}(t,x, \xi) = \bs \lambda_1[\xi](\bs u(t,x)) \bs \chi_{\bs u}(t,x,\xi) 
$$
and the speed $\bs \lambda_1[\xi](u)$ is strictly monotone in $\xi$ (see Proposition \ref{prop:localspeed}). A similar monotonicity property holds for the second Riemann invariant in connection with the second equation \eqref{eq:kin221}. In the \say{local} case (i.e. $\bs \lambda_1[\xi](\bs u) \equiv \bs \lambda_1(\xi)$) it is known that, if the velocity is not constant in $\xi$, then one can use the dispersive properties of the transport term to obtain some regularity of the solution $\bs u$ \cite{DLM91}, \cite{LPT94a}. However, these results have not been successfully applied to nonlocal equations such as \eqref{eq:kin22}, \eqref{eq:kin221}. The present kinetic formulation, 
obtained in connection
with
the  Lagrangian representation recently developed in the context of scalar conservation laws (see \cite{BBM17}, \cite{Mar19})
 could be useful to study $BV$-like regularity properties of these solutions and will be a topic for future research.

Some remarks are here in order:
\begin{itemize}
    \item[-] Kinetic formulations that \emph{characterize} entropy solutions have been obtained for particular systems, see \cite{LPT94b} for the system of isentropic gas dynamics, or \cite{PT00} for a systems in elasticity. A generalization of the kinetic formulation for the system of isentropic gas dynamics with $\gamma = 3$ leads to the \textit{multibranch solutions} introduced by Brenier \& Corrias \cite{BC98}, which can be viewed as an example of kinetic formulations in the setting of a very specific system of $n$ conservation laws.  Equations \eqref{eq:kin22}, \eqref{eq:kin221} (without assumptions on the sign of $\mu_1, \nu_1$) do not characterize entropy solutions, but rather \emph{finite entropy solutions} (see Theorem \ref{thm:kin}). Since we do not assume any specific structure on the system, the task of characterizing exactly the class of entropy solutions at the kinetic level seems a challenging topic.

    \item[-] When considering the physical viscosity, as e.g. in \cite{CP10} for the system of gas dynamics, vanishing viscosity solutions might not have a signed dissipation measure for every convex entropy. Therefore they might be a priori only finite entropy solutions, satisfying kinetic formulations similar to the one in Theorem \ref{thm:kin}.

\item[-] The kinetic formulation of Theorem \ref{thm:kin} contains  additional source terms $\mu_0, \nu_0$, which appear as the result of \say{decoupling} the conservation law into two kinetic equations associated with the two Riemann invariants. In \cite{ABB23} it is conjectured that solutions to $2\times2$ systems of conservation laws should share some of their regularity properties with scalar conservation laws with source terms, in particular this result seems to go in the same direction of \cite{ABB23}.
\end{itemize}

Combining the entropies of \cite{PT00} with the above mentioned Lagrangian tools
we establish the main result of this paper.
\begin{thm}[of Liouville-type]\label{thm:Liuv}
    Let $\bs u: \mathbb R^2 \to \mc U\subset \mathbb R^2$ be a bounded weak solution to a hyperbolic 
    system of two conservation laws \eqref{eq:systemi}. Assume that the eigenvalues are genuinely nonlinear:
    \begin{equation*}
        \partial_w\lambda_{1}(\bs u),\,
        \partial_z\lambda_{2}(\bs u)
        \geq \bar c >0\qquad \forall \; \bs u \in \mc U,
    \end{equation*}
    and
    that for every entropy-entropy flux pair $\eta, q$
    \begin{equation}\label{eq:isecondi}
    \partial_t \eta(\bs u) + \partial_x q(\bs u) = 0 \qquad \text{in $\mathscr D^\prime_{t,x}$}.
    \end{equation}
    Then $\bs u$ is a constant function.
\end{thm}

This result holds for any weak isentropic  (i.e. satisfying \eqref{eq:isecondi} for every entropy-entropy flux pair) solution, regardless of whether a uniformly  convex entropy exists, since only the kinetic formulation of Theorem \ref{thm:kin} is used, and isentropic solutions automatically satisfy \eqref{eq:kin22}, \eqref{eq:kin221} with $\mu_i = \nu_i = 0$.

A quite standard consequence of Theorem \ref{thm:Liuv} is that, for any finite entropy solution, there is a set $\br J \subset \mathbb R^+ \times \mathbb R$ with Hausdorff dimension at most $1$, such that every point in $\br J^c$ is a point of \textit{vanishing mean oscillation} (VMO), see Theorem \ref{thm:VMO}. It was known that such a Liouville-type theorem would have implied the VMO property (see, e.g., \cite{DOW03}, \cite{CT11}), but a proof of Theorem \ref{thm:Liuv} had been missing for some time. The 1-rectifiability of $\br J$ remains a challenging problem. Notice that $\br J$ can be defined for a finite entropy solution and it takes the form:
 \begin{equation}\label{eq:Jdef}
 \br J \doteq \Bigg\{(t,x) \in \Omega \; \Big| \; \limsup_{r \to 0^+} \frac{\bs \nu(B_r(t,x)) }{r} > 0  \Bigg\}
 \end{equation}
where 
$$
\bs \nu \doteq \bigvee_{\substack{\eta \in \mc E \\ |\eta|_{C^2} < 1}} \mu_{\eta} \; \in \; \msc M(\Omega).
$$
Here $\bigvee$ denotes the supremum in the sense of measures (see \cite[Definition 1.68]{AFP00}) and $\mc E$ is the set of entropies $\eta : \mc U \to \mathbb R$ (Definition \ref{def:eefp}), while $\mu_\eta$ is the corresponding dissipation measure in~\eqref{eq:entropydiss}.

\begin{remark}
    The measures $\mu_1, \mu_0$ (and $\nu_1, \nu_0$) are not uniquely determined by the left hand sides of \eqref{eq:kin22}, \eqref{eq:kin221}.
\end{remark}

The paper is structured as follows.

In Section \ref{sec:preliminaries} we introduce some preliminaries related to the general theory of hyperbolic conservation laws. 

In Section \ref{sec:EKF} we first recall the construction of \cite{PT00} and then we prove Theorem \ref{thm:kin}. 

In Section \ref{sec:LTKS} we introduce some tools related to the Lagrangian representation needed for the proof of Theorem \ref{thm:Liuv}.

Finally, in Section \ref{sec:LTVMO} we prove Theorem \ref{thm:Liuv}.

\section{Preliminaries About Conservation Laws}\label{sec:preliminaries}
We consider systems of two conservation laws
\begin{equation}\label{eq:system}
\partial_t\,\bs  u(t,x) + \partial_x \, f(\bs  u(t,x)) = 0, \qquad (t,x) \in \Omega \subset \mathbb R^+ \times \mathbb R, \qquad \bs u \in \mc U
\end{equation}
where $\mc U\subset \mathbb R^2$ is an open bounded set, $\bs u = (u_1, u_2) \in \mc U \subset \mathbb R^2$ is a state vector of conserved quantities, the flux $f$ is a smooth function $f: \mc U \to \mathbb R^2$. A typical choice for the domain $\Omega$ is $\Omega = \mathbb R^+\times \mathbb R$, although in this paper other domains are occasionally used. The system \eqref{eq:system} is called strictly hyperbolic if the matrix $\mr D f$ has distinct real eigenvalues 
$$
\lambda_1(\bs u) < \lambda_2(\bs u) \qquad \forall \; \bs u \in \mc U
$$
with corresponding eigenvectors $r_1(\bs u), r_2(\bs u)$. We also let $\ell_1, \ell_2$ be the corresponding left eigenvectors, normalized so that 
$$
\ell_i(\bs u) \cdot r_i(\bs u) = \delta_{i,j} \qquad \forall \; \bs u \in \mc U.
$$
    Being a system of two equations, \eqref{eq:system} admits a coordinate system of Riemann invariants $\phi_1, \phi_2$. We assume that the latter are smooth invertible functions $\phi = (\phi_1, \phi_2 ): \mc U \to \mathbb R^2$ defined by 
\begin{equation}\label{eq:riemdef}
    \nabla \phi_1(\bs u) = \ell_1(\bs u), \qquad \nabla \phi_2(\bs u) = \ell_2(\bs u) \qquad \forall \bs u \in \mc U.
\end{equation}
We let $\mc W \doteq (\phi_1, \phi_2)(\mc U) \subset \mathbb R^2$.
A function \(g\) can be expressed in terms of the state vector \(\bs u\) or in terms of the Riemann invariants \((w, z)\), according to 
\[
g(\bs u) = \hat g(\phi^{-1}(\bs u)), \qquad
\partial_w \hat g = r_1 \cdot \nabla g, \quad \partial_z \hat g = r_2 \cdot \nabla g.
\]
From now on, relying on a common abuse of notation, we will use the same symbol \(g\) for both expressions.

It is well known that weak solutions to hyperbolic systems of conservation laws are not unique, therefore in order to select physically relevant solutions, one is usually interested only in entropic solutions of \eqref{eq:system}. 
\begin{defi}[Entropies]\label{def:eefp}
    A pair of Lipschitz functions $\eta, q: \mc U \subset \mathbb R^2 \to \mathbb R$ is called an \textit{entropy-entropy flux pair} for \eqref{eq:system} if
    \begin{equation}\label{eq:entropyeq}
        \nabla \eta(\bs u) \cdot \mr D f(\bs u) = \nabla q(\bs u) \qquad \text{for almost every $\bs u \in \mc U$}.
    \end{equation}
\end{defi}
In the following we will also use weaker notions of entropy-entropy flux pairs.

 Admissible (entropy) solutions of \eqref{eq:system} are the ones that dissipate the family of \textit{convex} entropies. Precisely,
     a function $\bs u: \Omega \to \mc U$ is called an \textit{entropy weak solution} of \eqref{eq:system} if it satisfies 
 \begin{equation}
     \partial_t \eta + \partial_x q \leq 0 \qquad \text{in} \; \msc D^{\prime}(\Omega)
 \end{equation}
 for all entropy-entropy flux pairs $(\eta, q)$ with $\eta$ a convex function.
 The relevance of this definition lies in the fact that the viscous approximations to \eqref{eq:system}
 \begin{equation}\label{eq:vsystem}
 \bs u^\epsilon(t,x) + f(\bs u^\epsilon(t,x))_x = \epsilon \bs u^\epsilon_{xx},  \qquad (t,x) \in \Omega \subset \mathbb R^+ \times \mathbb R   , \qquad \bs u \in \mc U
 \end{equation}
 produce entropy admissible weak solutions of \eqref{eq:system} in the limit $\eps \to 0^+$. We say that $\bs u: \Omega \to \mc U$ is a \emph{vanishing viscosity solution} to \eqref{eq:system} if there exists a sequence $\eps_i \to 0^+$ and a sequence $\bs u^{\eps_i}: \Omega \to \mc U$ of solutions to \eqref{eq:vsystem} such that ${\bs u}^{\eps_i} \to \bs u$ in  $\mathbf L^1_{\mr{loc}}(\Omega)$.

We have the following well known energy bound, see e.g. \cite[Section 9.2]{Ser00}. Assuming the existence of a uniformly convex entropy $E : \mc U \to \mathbb R$, it follows that if $\bs u^\eps$ is a family of solutions to \eqref{eq:vsystem} with $\bs u^\eps(t,x) \in \mc U$, then 
for every compact set $K \subset \Omega$ there is a constant $C_K$ such that
     \begin{equation}\label{eq:energybound}
              \sup_{\eps>0} \iint_K \big(\sqrt{\epsilon} \,\bs u^\epsilon_x\big)^2 \dif x \dif t \leq C_K.
          \end{equation}
In fact, more precisely for every $M, T > 0$, there holds
\begin{equation}\label{eq:energybound1}
    \int_{0}^T \int_{-M - L(T-t)}^{M + L(T-t)} \big(\sqrt{\epsilon} \,\bs u^\epsilon_x\big)^2 \dif x \dif t \leq C \int_{-M-LT}^{M + LT} E(\bs u(0, x)) \dif x
\end{equation}
where $C, L> 0$ are positive constants depending only on $f$ and $E$.

\begin{defi}\label{defi:GNL}
We say that the eigenvalue $\lambda_1$ ($\lambda_2$) is genuinely nonlinear (GNL) if there is $\bar c > 0$ such that 
$$
\partial_w\lambda_{1}(\bs u)\geq \bar c \quad \Big( \partial_z\lambda_{2}(\bs u) \geq \bar c\Big)\qquad \forall \; \bs u \in \mc U.
$$
\end{defi}


\section{Entropies, Kinetic Formulation}\label{sec:EKF}

\subsection{Construction of Singular Entropies}\label{sec:singen}
In this subsection we recall the construction of singular entropies performed in \cite{PT00}, \cite{Tza03}.
We employ a relaxed (with respect to Definition \ref{def:eefp})  concept of entropy-entropy flux pair.
In particular, a \emph{weak entropy-entropy flux pair} is a pair of functions $\eta, q: \mc U \to \mathbb R$ that solves in the sense of distribution
\begin{equation}\label{e:wentropy}
\nabla q(\bs u) \, - \, \mr D f (\bs u) \, \nabla \eta(\bs u) = 0 \qquad \text{in $\msc D^\prime (\mc U)$}.   
\end{equation}
Let $g, h$ be the unique solutions to 
\begin{equation}\label{eq:hgdef}
h_w = \frac{\lambda_{2w}}{\lambda_1-\lambda_2}h, \qquad g_z = -\frac{\lambda_{1z}}{\lambda_1-\lambda_2}g, \qquad h({\ubar w}, z) = 1, \qquad g(w, {\ubar z}) = 1.
\end{equation}
They can be computed explicitly as
$$
\begin{aligned}
g(w, z) & = \exp\left[\int_{{\ubar z}}^{ z} -\frac{\lambda_{1z}(w, y)}{\lambda_1(w, y)-\lambda_2(w, y)}\dif y  \right] \\
h(w, z) & = \exp\left[\int_{{\ubar w}}^{w} \frac{\lambda_{2w}(y, z)}{\lambda_1(y, z)-\lambda_2(y, z)}\dif y  \right].
\end{aligned}
$$
and they are uniformly positive on $\mc W$. 
It is then classical (see e.g. \cite[Section 9.3]{Ser00}) that $\eta$ is a smooth entropy if and only if, in Riemann coordinates,
$$
\eta_{wz} = \frac{g_z}{g}\eta_w + \frac{h_w}{h} \eta_z \qquad \text{in $\mc  W$}.
$$
Following \cite{PT00}, we first construct a family of \emph{smooth} entropies $\bs \Theta[\xi, b_0](w,z)$, depending on two parameters: a scalar $\xi \in[{\ubar w}, \bar w]$ and a smooth function $b_0 : [{\ubar w}, \bar w] \to \mathbb R$. These entropies are constructed so that they can be \say{cut} along a line $\{w = \xi\}$. By this we mean that 
\begin{equation}\label{e:chir}
\bs \chi[\xi, b_0](w, z) \doteq \bs \Theta[\xi, b_0](w,z)\cdot  \mathbf 1_{\{w \geq \xi\}}(w, z)
\end{equation} 
and
\begin{equation}\label{e:chil}
\wt{\bs \chi}[\xi, b_0](w, z) \doteq \bs \Theta[\xi, b_0](w,z)\cdot  \mathbf 1_{\{w \leq \xi\}}(w, z)
\end{equation}
will still be (discontinuous) weak entropies.
\begin{defi}
We denote by $\bs \Theta[\xi, b_0]$ the entropy constructed as the unique solution to the Goursat-boundary value problem (see Figure \ref{fig:goursat})
$$
\begin{cases}
\bs \Theta_{wz} = \frac{g_z}{g}\bs \Theta_w + \frac{h_w}{h} \bs \Theta_z, & \text{in} \; \mc W\\
\\
\bs \Theta(w, {\ubar z}) = b_0(w), & \forall \; w \in [{\ubar w}, \bar w] \\
\\
\bs \Theta(\xi, z) = b_0(\xi)g(\xi, z) & \forall \; z \in [{\ubar z}, \bar z].
\end{cases}
$$
\end{defi}
Since $g(\xi, {\ubar z}) = 1$, the two boundary conditions are compatible (continuous) at the point $(\xi, {\ubar z})$, by construction. Since $h, g$ are smooth and bounded away from zero, the existence of a unique, smooth solution $\bs \Theta$ to the above boundary value problem is standard. For a proof of this fact, see e.g. \cite[Section 9.3]{Ser00}, in which it is proved that solutions to the Goursat problem are at least as smooth as the data and as the coefficients $g_z/g, h_w/h$. Moreover, it also follows that 
$$
w, z, \xi \mapsto \bs \Theta[\xi, b_0](w,z)
$$
is smooth as a function of three variables $w, z, \xi$.
Now for fixed $\xi, b_0$ we consider the entropy flux $\bs \Xi \equiv \bs \Xi[\xi, b_0]$ associated with $\bs \Theta\equiv \bs \Theta[\xi, b_0]$: we have that 
$$
\bs \Xi_z(\xi, z) = \lambda_2(\xi, z) \bs \Theta_z(\xi, z) = -\frac{\lambda_2(\xi, z) \lambda_{1z}(\xi, z)}{\lambda_1(\xi, z)-\lambda_2(\xi, z)}\bs \Theta(\xi, z)  = (\lambda_1(\xi, z)\bs \Theta(\xi, z))_z
$$
where the first equality follows from applying \eqref{eq:entropyeq} to $\bs \Theta, \bs \Xi$, and by taking the scalar product with $r_2(\bs u)$, while the second equality follows from the fact that $\bs \Theta(\xi, z) = b_0(\xi) g(\xi, z)$ for every $z \in [\ubar z, \bar z]$ and by \eqref{eq:hgdef}.
Therefore up to an additive constant in the entropy flux we can assume that 
\begin{equation}\label{eq:efluxxi}
\bs \Xi(\xi, z) = \lambda_1(\xi, z) \bs \Theta(\xi, z) \qquad \forall \; z \in [{\ubar z}, \bar z].
\end{equation}
Thanks to \eqref{eq:efluxxi}, we see that $(\bs \chi, \bs \psi)$, and $(\wt{\bs \chi}, \wt{\bs \psi})$, where
\begin{equation}\label{eq:fluxes}
\bs \psi[\xi, b_0] \doteq \bs \Xi[\xi, b_0]\cdot  \mathbf 1_{\{w \geq \xi\}}, \quad \wt {\bs \psi}[\xi, b_0] \doteq \bs \Xi[\xi, b_0]\cdot  \mathbf 1_{\{w \leq \xi\}}
\end{equation}
are entropy-entropy flux pair solving \eqref{e:wentropy}.

\begin{figure}
    \centering

\tikzset{every picture/.style={line width=0.75pt}} 

\begin{tikzpicture}[x=0.75pt,y=0.75pt,yscale=-0.8,xscale=0.8]

\draw    (110,250) -- (447,250) ;
\draw [shift={(450,250)}, rotate = 180] [fill={rgb, 255:red, 0; green, 0; blue, 0 }  ][line width=0.08]  [draw opacity=0] (5.36,-2.57) -- (0,0) -- (5.36,2.57) -- cycle    ;
\draw    (270,60) -- (270,250) ;
\draw    (420,60) -- (420,260) ;
\draw    (120,43) -- (120,260) ;
\draw [shift={(120,40)}, rotate = 90] [fill={rgb, 255:red, 0; green, 0; blue, 0 }  ][line width=0.08]  [draw opacity=0] (5.36,-2.57) -- (0,0) -- (5.36,2.57) -- cycle    ;
\draw    (120,60) -- (420,60) ;
\draw    (370,240) .. controls (398.76,207.69) and (430.24,211.74) .. (458.7,181.41) ;
\draw [shift={(460,180)}, rotate = 132.11] [color={rgb, 255:red, 0; green, 0; blue, 0 }  ][line width=0.75]    (6.56,-1.97) .. controls (4.17,-0.84) and (1.99,-0.18) .. (0,0) .. controls (1.99,0.18) and (4.17,0.84) .. (6.56,1.97)   ;
\draw [line width=1.5]    (120,250) -- (420,250) ;
\draw [line width=1.5]    (270,60) -- (270,250) ;

\draw (267,262.4) node [anchor=north west][inner sep=0.75pt]  [font=\footnotesize]  {$\xi $};
\draw (465,154.4) node [anchor=north west][inner sep=0.75pt]  [font=\footnotesize]  {$b_{0}$};
\draw (241,24.4) node [anchor=north west][inner sep=0.75pt]  [font=\footnotesize]  {$\boldsymbol{\Theta }[ \xi ,\ b_{0}]$};
\draw (179,142.4) node [anchor=north west][inner sep=0.75pt]  [font=\footnotesize]  {$w\ \leq \xi $};
\draw (329,142.4) node [anchor=north west][inner sep=0.75pt]  [font=\footnotesize]  {$w\ \geq \xi $};
\draw (87,242.4) node [anchor=north west][inner sep=0.75pt]  [font=\footnotesize]  {$\underline{z}$};
\draw (91,52.4) node [anchor=north west][inner sep=0.75pt]  [font=\footnotesize]  {$\overline{z}$};
\draw (117,272.4) node [anchor=north west][inner sep=0.75pt]  [font=\footnotesize]  {$\underline{w}$};
\draw (416,272.4) node [anchor=north west][inner sep=0.75pt]  [font=\footnotesize]  {$\overline{w}$};

\end{tikzpicture}

    \caption{Goursat problem for the entropy $\bs \Theta[\xi, b_0]$. The data are given along the thick lines.}
    \label{fig:goursat}
\end{figure}
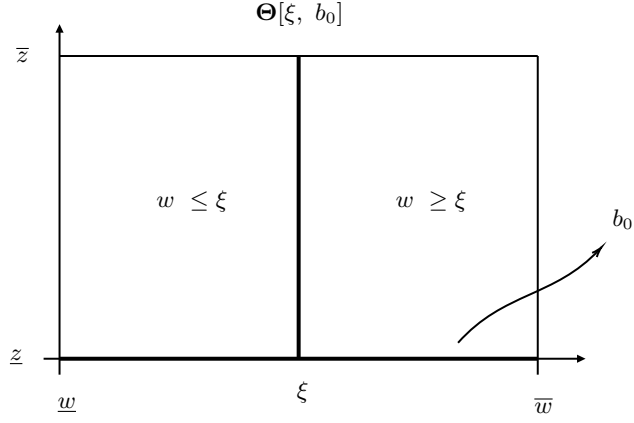

The entropies $\bs \Theta[\xi, b_0]$ depend on a number $\xi \in [{\ubar w}, \bar w]$ and on a function $b_0$. To obtain a \say{one dimensional} kinetic formulation for the first Riemann Invariant, for every $\xi$ we need to make a choice  of $b_0$.  
Following \cite{PT00}, we choose
$$
b_0(w) = 1 \qquad \forall \; w \in [{\ubar w}, \bar w]
$$
and with this choice we rename the entropy $\bs \Theta$ omitting the dependence on $b_0$, which is now fixed:
$$
\bs \Theta[\xi](w, z) \equiv \bs \Theta[\xi, 1](w,z) \qquad \forall \; \xi \in [{\ubar w}, \bar w]
$$
and the same for $\bs \chi[\xi], \bs \psi[\xi] \equiv \bs \chi[\xi,1], \bs \psi[\xi, 1]$.

The following proposition contains some structural results for the entropies $\bs \chi$.
\begin{prop}\label{prop:localspeed}
    There exists positive $\bar {r}, c > 0$  such that, for every $\xi, w \in [\ubar w, \bar w]$  and $z \in [\ubar z, \bar z]$ such that $\xi \leq w \leq \xi + \bar r$, the following holds:
\begin{enumerate}
    \item Strict positivity of the entropies: 
    $$
    \bs \chi[\xi](w, z) \geq c > 0
    $$
    \item If $\lambda_{1}$ is genuinely nonlinear, then we have the monotonicity of the kinetic speed:
$$
\frac{\dif}{\dif \xi} \bs \lambda_1[\xi](w, z) \geq c > 0 
$$
where
\begin{equation}\label{eq:lambdadef}
\bs \lambda_1[\xi](w, z) \doteq \frac{\bs \psi[\xi](w, z)}{\bs \chi[\xi](w, z)} \qquad   \forall \; \xi \leq w \leq \xi  +\bar r.
\end{equation}
\end{enumerate}
\end{prop}

\begin{proof}
    Fix $\xi$. Since the entropy $\bs \chi[\xi]$ is uniformly positive along the boundary data curve $\{(w,z) \in \mc W \; | \; w = \xi\}$, there exists $\delta(\xi) > 0, c_1 > 0$ such that 
 $$
    \bs \chi[\xi](w, z) \geq c_1 > 0, \qquad \forall \; (w, z) \in \mc W, \qquad \xi \leq w \leq \xi  +\bar r.
    $$
   Then, since the function $(\xi, w, z) \mapsto \bs \Theta[\xi](w,z)$ is in particular continuous and since $\xi \in [\ubar w, \bar w]$ which is compact, there exists uniform $r, c > 0$ (not dependent on $\xi$) such that (1) holds.
 Furthermore, for every $w \geq \xi$, the entropy flux $\bs \psi[\xi]$ associated to $\bs \chi[\xi]$ can be computed as
\begin{equation}\label{eq:entflux}
\begin{aligned}
\bs \psi[\xi](w, z) & = \lambda_1(w, z) \bs \chi[\xi](\xi, z) + \int_\xi^w \lambda_1(v, z) \bs \chi_w[\xi](v, z) \dif v \\
& = \lambda_1(w, z) \bs \chi[\xi](w, z) - \int_\xi^w \lambda_{1w}(v, z) \bs \chi[\xi](v, z) \dif v.
\end{aligned}
\end{equation}
where the first equality follows from the fundamental theorem of calculus and \eqref{eq:entropyeq}, and the second follows by integrating by parts.
Therefore, if the first eigenvalue is genuinely nonlinear (Definition \ref{defi:GNL}) the kinetic speed $
\bs \lambda_1[\xi](w, z)$ 
is monotonically increasing in $\xi$ if $\xi$ is close to $w$: in particular,
for some $c_2 > 0$
$$
\frac{\dif}{\dif \xi} \bs \lambda_1[\xi](w, z) \geq c_2 > 0 \qquad \forall \; (w, z) \in \mc W, \qquad \xi \leq w \leq \xi  +\bar r.
$$
The existence of uniform $r, c$ such that (2) holds is again ensured by the smoothness of all the functions involved.
\end{proof}

A completely symmetric construction can be made for entropies that can be cut along the second Riemann invariant; for these entropies, for $\zeta \in [{\ubar z}, \bar z]$, we let $\bs \upsilon[\zeta](w,z)$ be entropy corresponding to $\bs \chi[\xi](w,z)$, and $\bs \varphi[\zeta](w,z)$ for the respective entropy flux, corresponding to $\bs \psi[\xi](w,z)$.

\subsection{Kinetic Formulation}
We can now prove the result of this section, Theorem \ref{thm:kin}, according to which a function $\bs u: \Omega \to \mc U$ is a finite entropy solution if and only if it satisfies a suitable pair of kinetic-type equations. In the following, given a function $\bs u: \Omega \to \mc U$, we define the bounded function
\begin{equation}\label{eq:kfdef}
\begin{aligned}
    & \bs \chi_{\bs u}(t,x,\xi) \doteq \bs \chi[\xi](\bs u(t,x)) \qquad   \forall \; (t, x, \xi) \in \Omega \times (\ubar w, \bar w), \\
    & \bs \upsilon_{\bs u}(t,x,\zeta) \doteq \bs \upsilon[\zeta](\bs u(t,x)) \qquad   \forall \; (t, x, \zeta) \in \Omega \times (\ubar z, \bar z)
    \end{aligned}
\end{equation}
\begin{equation}\label{eq:kfdef1}
\begin{aligned}
    & \bs \psi_{\bs u}(t,x,\xi) \doteq \bs \psi[\xi](\bs u(t,x)) \qquad   \forall \; (t, x, \xi) \in \Omega \times (\ubar w, \bar w), \\
    & \bs \varphi_{\bs u}(t,x,\zeta) \doteq \bs \varphi[\zeta](\bs u(t,x)) \qquad   \forall \; (t, x, \zeta) \in \Omega \times (\ubar z, \bar z)
    \end{aligned}
\end{equation}
where $\bs \chi, \bs \psi$ and $\bs \upsilon, \bs \varphi$ are the discontinuous entropy-entropy flux pairs defined in Subsection \ref{sec:singen}.

\begin{remark}\label{rem:prop:kfeseqsym}
Defining
\begin{equation}\label{eq:kfdeftilde}
\begin{aligned}
    & \widetilde{\bs \chi}_{\bs u}(t,x,\xi) \doteq \widetilde{\bs \chi}[\xi](\bs u(t,x)), \quad  \widetilde{\bs \psi}_{\bs u}(t,x,\xi) \doteq \widetilde{\bs \psi}[\zeta](\bs u(t,x)), \quad   \forall \; (t, x, \xi) \in \Omega \times (\ubar w, \bar w)
    \end{aligned}
\end{equation}
and with obvious notation also $\widetilde{ \bs \upsilon}_{\bs u}$, $\widetilde{ \bs \varphi}_{\bs u}$
one can make a symmetric statement to the one in Theorem \ref{thm:kin}: in particular,
$\bs u$ is an isentropic solution if and only if (recall \eqref{eq:kfdeftilde})
    \begin{equation}
        \partial_t \widetilde{\bs \chi}_{\bs u}(t,x,\xi) + \partial_x \widetilde{\bs \psi}_{\bs u}(t,x,\xi) = 0 \qquad \text{in $\msc D^\prime$}\big(\Omega \times \mathbb R \big)
    \end{equation}
     \begin{equation}
        \partial_t \widetilde{\bs \upsilon}_{\bs u}(t, x, \zeta)+ \partial_x \widetilde{\bs \varphi}_{\bs u}(t, x, \zeta) =0 \qquad \text{in $\msc D^\prime$}\big(\Omega \times  \mathbb R \big)
    \end{equation}
    Notice that $\wt {\bs \chi}_{\bs u}, \wt {\bs \psi}_{\bs u}$ are now supported on the epigraph $\{\xi \geq w(t,x)\}$ of the first Riemann invariant (recall equation \ref{e:chil}), while $\wt {\bs \upsilon}_{\bs u}$ and $\wt {\bs \varphi}_{\bs u}$ are  supported on the epigraph $\{\xi \geq z(t,x)\}$ of the second Riemann invariant. 
\end{remark}

We now prove the Theorem.
\begin{proof}[Proof of Theorem \ref{thm:kin}]
  {\bf 1.}  Assume that $\bs u$  satisfies \eqref{eq:kin22}, \eqref{eq:kin221}. Let $\eta, q \in C^2$ be any smooth entropy-entropy flux pair, and without loss of generality assume that $\eta(\ubar w, \ubar z) = 0 = q(\ubar w, \ubar z)$. By the representation formula of \cite[Theorem 3.4]{PT00} we have that (recalling  the construction of singular entropies in Section \ref{sec:singen})
    \begin{equation}\label{eq:entsup}
    \begin{aligned}
\eta(\bs u) & \doteq \int_{\ubar w}^{\bar w} \bs \chi[\xi] (\bs u)\rho_1(\xi) \dif \xi  +  \int_{\ubar z}^{\bar z} \bs \upsilon[\zeta](\bs u) \rho_2(\zeta) \dif \zeta \\
q(\bs u)  & \doteq \int_{\ubar w}^{\bar w} \bs \psi[\xi] (\bs u)\rho_1(\xi) \dif \xi  +  \int_{\ubar z}^{\bar z} \bs \varphi[\zeta](\bs u) \rho_2(\zeta) \dif \zeta
    \end{aligned}
    \end{equation}
where
$
\varrho_1(\xi)  \doteq \frac{\dif}{\dif \xi} \eta(\xi, \ubar z), \varrho_2(\zeta) \doteq \frac{\dif}{\dif \zeta} \eta(\ubar w, \zeta)  \in C^1.
$
Then we obtain 
$$
\begin{aligned}
    \partial_t \eta(\bs u) + \partial_x q (\bs u) & = \int_{\ubar w}^{\bar w} \varrho_1(\xi) \dif \mu_0(\xi, x, t) - \int_{\ubar w}^{\bar w} \varrho^\prime_1(\xi) \dif \mu_1(\xi, x, t) \\
    & +  \int_{\ubar z}^{\bar z} \varrho_2(\zeta) \dif \nu_0(\zeta, x, t)- \int_{\ubar z}^{\bar z} \varrho^\prime_2(\zeta) \dif \nu_1(\zeta, x, t) \in \msc M(\Omega)
\end{aligned}
$$
where here, if $\gamma(v, x,t) \in \mathscr M(\mathbb R\times \Omega)$ and $\rho(v)$ is a smooth function, we denote by $\int \rho(v) \dif \gamma(v, x, t) \in \mathscr M(\Omega)$, with a slight abuse of notation, the measure defined by
$$
\int_\Omega \varphi(t,x)  \dif \Big(\int \rho(v) \dif \gamma(v, \cdot, \cdot)  \Big)(t,x) : = \int_{\Omega \times \mathbb R} \varphi(t,x) \rho(v) \dif \gamma(v, x, t). 
$$

{\vspace{0.3cm}}
{\bf 2.} Conversely, assume that $\bs u$ is a finite entropy solution.  Define a distribution $T \in \msc D^\prime(\Omega \times (\ubar w, \bar w))$ by 
 $$
 \langle T, \varphi\, \varrho \rangle \doteq \int_{\Omega} \varphi_t \eta_{\varrho}(\bs u)\,  +\, \nabla_x \varphi\cdot    q_\varrho(\bs u) \dif x \dif t \qquad \forall \phi \in C^1_c(\Omega), \quad \varrho \in C^1_c((\ubar w, \bar w))
 $$
 where we define the entropy-entropy flux pair associated to $\varrho$
 $$
 \eta_{\varrho}(\bs u) \doteq \int_{\ubar w}^{\bar w} \rho(\xi) \bs \chi[\xi](\bs u) \dif \xi   , \qquad \bs q_\varrho(\bs u) \doteq \int_{\ubar w}^{\bar w} \rho(\xi) \bs \psi[\xi](\bs u) \dif \xi  .
 $$
Consider any open set $U$ compactly contained in $\Omega$ and for any $\varphi \in \msc D(U)$ we define a linear functional $L_\varphi \, : \,  C(\mathbb R) \to \mathbb R$ by 
$$
L_\varphi(\varrho) \doteq   \int_{U} \varphi_t \eta_{\varrho}(\bs u)\,  +\, \nabla_x \varphi\cdot    q_\varrho(\bs u) \dif x \dif t.
$$
Each functional $L_\varphi$ is bounded, and therefore also continuous, since it holds
$$
|L_\varphi(\varrho)| \leq C_{U, \varphi} \, \Vert \varrho\Vert_{\mc C^0} 
$$
for some constant $C_{U, \varphi}$ depending only the set $U$ and the $C^1$ norm of the function $\varphi$.
Since $\bs u$ is a finite entropy solution, we deduce that the family  of functionals $L_\varphi$ is pointwisely bounded on $C^1$, because
$$
\sup_{\substack{\varphi \in \msc D(U) \\ |\varphi| \leq 1} } |L_\varphi(\varrho)| \leq \int_U \dif |\mu_{\eta_\varrho}| \qquad \forall \varrho \in C^1.
$$
Therefore, by the uniform boundedness principle, the family $L_\varphi$ is uniformly (norm) bounded, that is
\begin{equation}\label{eq:ubp}
    \sup_{\substack{\varphi \in \msc D(U) \\ \|\varphi\|_{C^0} \leq 1, \,\|\varrho\|_{C^1}  \leq 1} } |L_\varphi(\varrho)| = \sup_{\substack{\varphi \in \msc D(U) \\ \|\varphi\|_{C^0} \leq 1, \,\|\varrho\|_{C^1} \leq 1} } |\langle T, \varphi \varrho\rangle|  \leq C_U.
\end{equation}
Therefore we obtained that the distribution $T$ satisfies the bounds
$$
\left|\langle T, \, \varphi \varrho\rangle \right| \leq C_U \big(\|\varphi\|_{C^0} + \|\varrho\|_{C^1}\big) \qquad \forall \, \varphi \in C^0(U), \quad \varrho \in C^1(U).
$$
By a standard application of the Riesz representation theorem we thus obtain the existence of locally finite measures $\mu_1, \mu_0$ such that \eqref{eq:kin22} holds.

Finally, if $\bs u$ is isentropic, then the distribution $T$ defined in the previous step clearly satisfies $T = 0$, and this proves the result.
\end{proof}

\subsection{Vanishing viscosity solutions}
Here we prove that vanishing viscosity solutions enjoy the following additional properties, which will be useful for future applications. This also yields a different, more explicit derivation of the kinetic formulation, with a finer characterization of the dissipation measures $\mu_i, \nu_i$.

\begin{prop}\label{prop:vvkin}
If $\bs u : \Omega \to \mc U$ is a vanishing viscosity solution and a uniformly convex entropy exists, then $\mu_1$ and $\nu_1$ in \eqref{eq:kin22}, \eqref{eq:kin221} can be taken to be positive measures, and for some constant $C > 0$, we have
         \begin{equation}\label{eq:sourcebound}
 ({\mathtt p_{t,x}})_\sharp |\mu_0| +  ({\mathtt p_{t,x}})_\sharp |\nu_0| \leq C \, ({\mathtt p_{t,x}})_\sharp \mu_1 + ({\mathtt p_{t,x}})_\sharp \nu_1. 
 \end{equation}
Here $\mathtt p_{t,x}$ denotes the canonical projection on the $t,x$ variables. We recall that given a measurable map $f: X \to Y$ between measure spaces $X, Y$, for any $\mu \in X$ the pushforward measure $f_\sharp \mu \in \msc M(Y)$ is defined by 
$$
f_\sharp \mu(A) = \mu(f^{-1}(A)) \qquad \forall \; \text{measurable} \; A \subset Y.
$$
\end{prop}

\begin{proof}

\noindent \textbf{1.} For every smooth $\xi \mapsto \varrho(\xi)$, we can consider a smooth entropy $\eta_{\varrho}$ where the entropy $\bs \chi[\xi]$ appears with density $\varrho(\xi)$:
\begin{equation}\label{eq:etarho}
\eta_{\varrho}(\bs u) \doteq \int_{\mathbb R}\bs \chi[\xi](\bs u) \varrho(\xi) \dif \xi, \qquad q_{\varrho}(\bs u) \doteq \int_{\mathbb R} \bs \psi[\xi](\bs u)  \varrho(\xi) \dif \xi.
\end{equation}
Then $\eta_{\varrho}, q_{\varrho}$ is a smooth entropy-entropy flux pair. In fact, clearly is a solution of \eqref{e:wentropy}, since each $\bs \chi[\xi], \bs \psi[\xi]$ is, and the equation is linear. The fact that it is smooth comes from the fact that $\varrho$ is smooth since an explicit calculation yields that the gradient of $\eta_{\varrho}$ is
$$
\nabla \eta_{\varrho}(\bs u) = \int_{\ubar w}^{{\phi_1(\bs u)}}  \nabla \bs \Theta[\xi](\bs u)\varrho(\xi) \dif \xi + \varrho({\phi_1(\bs u)})\cdot \bs \Theta [{\phi_1(\bs u)}](\bs u)\cdot  \nabla {\phi_1(\bs u)}  \qquad \forall \; \bs u \in \mc U
$$
where we recall that $\phi_1 : \mc U \to [\ubar w, \bar w]$ is defined in \eqref{eq:riemdef}.

\vspace{0.3cm}
\noindent \textbf{2.}
Now multiply from the left equation \eqref{eq:vsystem} by $\nabla \eta_{\varrho}(\bs u^\eps)$ to obtain 
\begin{equation}
    \begin{aligned}
        \nabla \eta_{\varrho}(\bs u^\eps)\big[\partial_t {\bs u^\eps} + f(\partial_x {\bs u^\eps}) \big] & = \eps \nabla \eta_{\varrho}(\bs u) {\partial^2_{xx}\bs u^\eps} \\
        & = \eps\int_{\ubar w}^{{\phi_1 (\bs u^\eps)}}   \nabla \bs \Theta[\xi](\bs u)\varrho(\xi) \dif \xi\, \partial^2_{xx}{\bs u^\eps} \\
        & + \eps \varrho({\phi_1(\bs u^\eps) })\bs \Theta [{\phi_1 (\bs u^\eps)}](\bs u^\eps) \nabla {\phi_1(\bs u^\eps) }{\partial^2_{xx} \bs u^\eps}
    \end{aligned}
\end{equation}
where from now on the symbol $\nabla$ will be reserved to denote the gradient of a function in the $\bs u$ variable.
We calculate the first term: 
\begin{equation}\label{eq:meps1}
\begin{aligned}
    \eps\int_{\ubar w}^{{\phi_1 (\bs u^\eps)}}   \nabla \bs \Theta[\xi](\bs u)\varrho(\xi) \dif \xi\,{\partial^2_{xx} \bs u^\eps} & = \partial_x \Big[\eps \int_{\ubar w}^{{\phi_1 (\bs u^\eps)}}   \nabla \bs \Theta[\xi]({\bs u^\eps}) \varrho(\xi) \dif \xi\,{\partial_x \bs u^\eps}\Big] \\
    & -\eps\partial_x \Big[ \int_{\ubar w}^{{\phi_1 (\bs u^\eps)}}   \nabla \bs \Theta[\xi]({\bs u^\eps})\varrho(\xi) \dif \xi\Big] {\partial_x \bs u^\eps}.
    \end{aligned}
\end{equation}
and the second term:
\begin{equation}\label{eq:meps2}
    \begin{aligned}
        \eps \varrho({\phi_1(\bs u^\eps})) \bs \Theta [{\phi_1 (\bs u^\eps)}](\bs u^\eps)\cdot  \nabla {\phi_1(\bs u^\eps) }{\partial^2_{xx} \bs u^\eps} & = \Big[\eps \varrho({\phi_1(\bs u^\eps) }) \bs \Theta [{\phi_1 (\bs u^\eps)}](\bs u^\eps)\cdot  \nabla {\phi_1(\bs u^\eps) }{\partial_x \bs u^\eps}\Big]_x \\
        & - \Big[\eps \varrho({\phi_1(\bs u^\eps) }) \bs \Theta [{\phi_1 (\bs u^\eps)}](\bs u^\eps)\cdot  \nabla {\phi_1(\bs u^\eps) }\Big]_x{\partial_x \bs u^\eps}.
    \end{aligned}
\end{equation}
The second term in the right hand side of \eqref{eq:meps2} can be calculated as 
\begin{equation}\label{eq:meps3}
\begin{aligned}
&  \Big[\eps \varrho({\phi_1(\bs u^\eps) })  \bs \Theta [{\phi_1 (\bs u^\eps)}](\bs u^\eps)\cdot \nabla {\phi_1(\bs u^\eps) }\Big]_x{\partial_x \bs u^\eps} \\
& =  \Big[\eps \varrho^\prime({\phi_1(\bs u^\eps) })\partial_x {\phi_1(\bs u^\eps) } \bs \Theta [{\phi_1 (\bs u^\eps)}](\bs u^\eps)\cdot \nabla {\phi_1(\bs u^\eps) }\Big ]\cdot {\partial_x \bs u^\eps} \\
 & +  \eps \varrho({\phi_1(\bs u^\eps) })\langle \mr D \Big(\bs \Theta [{\phi_1 (\bs u^\eps)}](\bs u^\eps)\cdot \nabla {\phi_1(\bs u^\eps) }\Big) {\partial_x \bs u^\eps}, \, {\partial_x \bs u^\eps}\rangle\\
 & =  \varrho^\prime({\phi_1(\bs u^\eps) }) \bs \Theta [{\phi_1 (\bs u^\eps)}](\bs u^\eps)\Big[ \sqrt{\eps}  \partial_x {\phi_1(\bs u^\eps) }\Big]^2 \\
 & +   \eps \varrho({\phi_1(\bs u^\eps) })\langle \mr D \Big(\bs \Theta [{\phi_1 (\bs u^\eps)}](\bs u^\eps)\cdot \nabla {\phi_1(\bs u^\eps) }\Big) {\partial_x \bs u^\eps}, \, {\partial_x \bs u^\eps}\rangle.
\end{aligned}
\end{equation}
Therefore we have 
\begin{equation}\label{eq:disscomplete}
    \begin{aligned}
        \partial_t \eta_{\varrho}(\bs u^\eps) + \partial_x q_\varrho(\bs u^\eps) & = \nabla \eta_{\varrho}(\bs u^\eps)\big[\partial_t{\bs u^\eps} + \partial_x f({\bs u^\eps}) \big]    \\
        & =  -\eps\partial_x \Big[ \int_{\ubar w}^{{\phi_1 (\bs u^\eps)}}   \nabla \bs \Theta[\xi]({\bs u^\eps})\varrho(\xi) \dif \xi\Big] {\partial_x \bs u^\eps} \\
        & -\varrho^\prime({\phi_1(\bs u^\eps) }) \bs \Theta [{\phi_1 (\bs u^\eps)}](\bs u^\eps)\Big[ \sqrt{\eps}  \partial_x {\phi_1(\bs u^\eps) }\Big]^2 \\
& -  \eps \varrho({\phi_1(\bs u^\eps) })\langle \mr D \Big(\bs \Theta [{\phi_1 (\bs u^\eps)}](\bs u^\eps)\cdot \nabla {\phi_1(\bs u^\eps) }\Big) {\partial_x \bs u^\eps}, \, {\partial_x \bs u^\eps}\rangle \\
         & + g^\eps_{\varrho}
    \end{aligned}
\end{equation}
where 
$$
g^\eps_{\varrho} \doteq \partial_x \Big[\eps \int_{\ubar w}^{{\phi_1 (\bs u^\eps)}}   \nabla \bs \Theta[\xi]({\bs u^\eps}) \varrho(\xi) \dif \xi\,{\partial_x \bs u^\eps}\Big]  +\partial_x \Big[\eps \varrho({\phi_1(\bs u^\eps) }) \bs \Theta [{\phi_1 (\bs u^\eps)}](\bs u^\eps)\cdot  \nabla {\phi_1(\bs u^\eps) }{\partial_x \bs u^\eps}\Big].
$$
We notice that $g^\eps_{\varrho}$ is going to zero in distributions as $\eps \to 0^+$. In fact, since $\bs u$ admits a uniformly convex entropy, using \eqref{eq:energybound}, we deduce that for every compact $K \subset \Omega$
$$
\begin{aligned}    
\left\|\eps \int_{\ubar w}^{{\phi_1 (\bs u^\eps)}}   \nabla \bs \Theta[\xi]({\bs u^\eps}) \varrho(\xi) \dif \xi\,{\partial_x \bs u^\eps}\right\|_{\mathbf L^1(K)} & = \mc O(1) \cdot \|\varrho\|_{\br C^0}\cdot \sqrt{\eps} \left\|\sqrt \eps \partial_x \bs u^\eps\right\|_{\mathbf L^1(K)}  \\
& = \mc O(1) \cdot \|\varrho\|_{\br C^0}\cdot \sqrt{\eps} \left\|\sqrt \eps \partial_x \bs u^\eps\right\|_{\mathbf L^2(K)}\\
& = \mc O(1) \cdot \|\varrho\|_{\br C^0}\cdot \sqrt{\eps} \cdot C_K^{\frac{1}{2}} \longrightarrow 0 \quad \text{as $\eps \to 0^+$}.
\end{aligned}
$$
where $\mc O(1)$ is a constant depending only on the compact set $K$. The same estimate shows that also the second term in $g^\eps$ is going to zero in distributions. 

\vspace{0.3cm}
\noindent \textbf{3.} Define the distribution $T^\eps \in \msc D^\prime(\Omega \times (\ubar w, \bar w))$
$$
\begin{aligned}
\langle T^\eps, \, \varphi \varrho \rangle \, & \doteq \, -\iiint_{\Omega \times \mathbb R}   [\partial_t \varphi(t,x)  \bs \chi_{\bs u^\eps}(t,x,\xi) \, + \, \partial_x \varphi(t,x) \bs \psi_{\bs u^\eps}(t,x,\xi)\big] \varrho(\xi) \dif \xi \, \dif x \dif t \\
& \quad = \iint_{\Omega} \varphi \big( \partial_t \eta_{\varrho}(\bs u^\eps) + \partial_x q_{\varrho} (\bs u^\eps) \dif x \dif t
\end{aligned}
$$
for all smooth $\varphi(t,x), \varrho(\xi)$ compactly supported $C^{\infty}$ functions, this is sufficient because finite sums $\sum_{i=1}^N \varphi_i(t,x) \varrho_i(\xi)$ are dense in $C_c^{\infty}(\mathbb R^3)$ (see e.g. \cite[Section 4.3]{FJ99}). Notice that since $\bs u^\eps \to\bs u$ in $\mathbf L^1_{loc}$, also $\bs \chi_{\bs u^\eps}, \bs \psi_{\bs u^\eps}$ converge in $\mathbf L^1_{loc}$ to $\bs \chi_{\bs u}, \bs \psi_{\bs u}$ and $T^\eps$ converges to the left hand side of \eqref{eq:kin22} in the sense of distributions. 

Thanks to \eqref{eq:disscomplete}, we have that 
$$
T^\eps = \mu_0^\eps \, + \, \partial_\xi \mu_1^\eps \, +f^\eps
$$
where $f^\eps$ is going to zero in distributions,  $\mu_0^\eps$, $\mu_1^\eps$ are locally uniformly bounded measures, and in particular:
\begin{enumerate}
\item  $f^\eps$ is defined by 
$$
\langle f^\eps, \varphi \varrho\rangle \doteq \langle g^\eps_{\varrho}, \varphi\rangle \qquad \forall \; \text{smooth} \; \varphi(t,x), \varrho(\xi);
$$
    \item $\mu_1^\eps$ accounts for the third line in \eqref{eq:disscomplete}  and 
    \begin{equation}\label{eq:defmueps1}
\mu_1^\eps \,\doteq  \, (\mr{id}, w^\eps)_{\sharp} \Big[ \bs \Theta [{\phi_1 (\bs u^\eps)}](\bs u^\eps)\big(\sqrt{\eps} \partial_x {\phi_1(\bs u^\eps) }  \big)^2\cdot \msc L^2\Big] \in \msc M^+_{t, x, \xi}
 \end{equation}
where 
$$
(\mr{id}, w^\eps) : \mathbb R^+ \times \mathbb R \to \mathbb R^+ \times \mathbb R\times [{\ubar w}, \bar w], \qquad (\mr{id}, w^\eps)(t,x) \doteq (t,x, w^\eps(t,x)).
$$
In particular $\mu_1^\eps$ is positive (because by definition $ \bs \Theta [{\phi_1 (\bs u^\eps)}](\bs u^\eps) > 0$) 
  and it satisfies the bound 
\begin{equation}\label{eq:mu1b}
|\mu_1^\eps|(K\times \mathbb R) \leq \sup |\bs \Theta| \cdot \sup |\nabla w|^2 \cdot \int_K \big(\sqrt{\eps} \partial_x \bs u^\eps\big)^2 \dif x \dif t = \mc O(1) \cdot C_K
\end{equation}
where $\mc O(1)$ is independent on $\eps$, and the last equality follows from \eqref{eq:energybound}.
\item $\mu_{0}^\eps$ accounts for the second and the forth lines of \eqref{eq:disscomplete} and is given by 
$$
\begin{aligned}
    \mu_0^\eps &  \doteq -\eps (\mr{id}, w^\eps)_\sharp \Big[\langle \mr D \Big(\bs \Theta [{\phi_1 (\bs u^\eps)}](\bs u^\eps)\cdot \nabla {\phi_1(\bs u^\eps) }\Big) {\partial_x \bs u^\eps}, \, {\partial_x \bs u^\eps}\rangle \\
    & + \langle \nabla \bs \Theta[{\phi_1 (\bs u^\eps)}](\bs u^\eps)\otimes \nabla {\phi_1 (\bs u^\eps)} \cdot \partial_x \bs u^\eps, \, \partial_x \bs u^\eps\rangle  \Big]\cdot \msc L^2 \\
    & -  \left\langle \nabla^2 {\bs \Theta}[\xi](\bs u^\eps) \partial_x \bs u^\eps, \, \partial_x \bs u^\eps \right\rangle \cdot \msc L^3 \llcorner \{\xi \geq {\phi_1 (\bs u^\eps)}\}. 
\end{aligned}
$$
\end{enumerate}
The same type of estimate leading to \eqref{eq:mu1b} shows also that
$$
|\mu^\eps_0|(K) = \mc O(1) \cdot C_K
$$
independently of $\eps$.
Therefore up to subsequences the measures $\mu_0^\eps$, $\mu_1^\eps$ weakly converge to limiting measures $\mu_0$ and $\mu_1\geq0$ that satisfy \eqref{eq:kin22}.
\end{proof}

We notice that the energy bound \eqref{eq:energybound1} translates into precise bounds for the kinetic measures.
\begin{coro}
    The measures $\mu_i, \nu_i$, $i = 0,1$ constructed in the proof of Theorem \ref{thm:kin} satisfy for all $M, T > 0$
    \begin{equation}
     \int_0^T \int_{-M-L(T-t)}^{M + L(T-t)} \int_{\mathbb R} \dif |\mu_i| \leq C \int_{-M-LT}^{M + LT} E(\bs u(0, x)) \dif x
    \end{equation}
    \begin{equation}
     \int_0^T \int_{-M-L(T-t)}^{M + L(T-t)}\int_{\mathbb R}  \dif |\nu_i|  \leq C \int_{-M-LT}^{M + LT} E(\bs u(0, x)) \dif x.
    \end{equation}
\end{coro}
\begin{proof}
    Recall that $\mu_1^\eps$ is the weak limit of a sequence $\{\mu_1^{\eps_k}\}_k$ defined in \eqref{eq:defmueps1}. Then using \eqref{eq:mu1b} with 
    $$
    K  := \{ (t,x) \; | \; x \in (-M-L(T-t), M + L(T-t)), \quad t \in (0, T) \}
    $$
    we obtain 
    $$
    \begin{aligned}
        \mu_1^\eps(K\times \mathbb R) \leq C \int_{K} (\sqrt{\eps} \partial_x \bs u^\eps)^2 \dif x \dif t \leq \int_{-M-LT}^{M+LT} E(\bs u(0, x)) \dif x
    \end{aligned}
    $$
    where in the last inequality we used \eqref{eq:energybound1}.  The corresponding inequality for $\nu_1$ is proved symmetrically. Finally, the inequalities for $\mu_0, \nu_0$ immediately follow from \eqref{eq:sourcebound}.
\end{proof}

\section{Lagrangian Tools in Kinetic Setting}\label{sec:LTKS}
In the following of this Section we assume that $\bs u$ is an isentropic solution to \eqref{eq:system} defined in $\Omega = \mathbb R^+ \times \mathbb R$. Then, it satisfies the kinetic formulation of Theorem \ref{thm:kin}, i.e. it satisfies \eqref{eq:kin22}, \eqref{eq:kin221} with $\mu_i = \nu_i = 0$. 
    We assume that $\bs u$ is a non constant function; in particular,
if $(\phi_1, \phi_2) : \mc U \to \mc W$ is the change of coordinates of the Riemann invariants, letting 
$$
w(t,x) = \phi_1(\bs u(t,x)), \qquad z(t, x) = \phi_2(\bs u(t,x))
$$
at least one of $w$, $z$ must be a non constant function. Therefore, from now on and without loss of generality, we assume that $w: \mathbb R^+\times \mathbb R \to [{\ubar w}, \bar w]$ is non constant. Then we have $w_{\min} < w_{\max}$ where 
\begin{equation}\label{eq:maxind}
w_{\max} \doteq \esssup_{t,x}  w
\end{equation}
\begin{equation}\label{eq:minind}
w_{\min} \doteq \essinf_{t,x} w 
\end{equation}
Let $\bar r, c > 0$ be fixed by Proposition \ref{prop:localspeed}; up to taking a smaller $r < \bar r$, we can, in addition to (1), (2) of Proposition \ref{prop:localspeed}, assume that $r$ satisfies also
\begin{equation}\label{eq:rchoice}
w_{\min} + r < w_{\max} - r.
\end{equation}
We define (recall Remark \ref{rem:prop:kfeseqsym}, and \eqref{eq:kfdeftilde})
\begin{equation}\label{eq:chibcdef}
\begin{aligned}
    \bs \chi^{\max}(t,x,\xi) \doteq \bs \chi_{\bs u}(t,x,\xi)\cdot \mathbf 1_{\{(t,x,\xi) \; | \; w_{\max}- r \leq \xi \leq  w_{\max}\}}(t,x,\xi)\\
     \bs \chi^{\min}(t,x,\xi) \doteq \bs {\wt\chi}_{\bs u}(t,x,\xi)\cdot \mathbf 1_{\{(t,x,\xi) \; | \; w_{\min}  \leq \xi \leq  w_{\min} +r\}}(t,x,\xi).
    \end{aligned}
\end{equation}
Recalling the definitions \eqref{e:chir}, \eqref{e:chil} and \eqref{eq:kfdef}, \eqref{eq:kfdeftilde}, notice that we have 
\begin{equation}
\mr{supp} \, \bs \chi^{\max} = \mr{hyp} \,  \phi_1(\bs u) \cap \Big(\mathbb R^+ \times \mathbb R \times (w^{\max}-r, w^{\max}) \Big)
\end{equation}
\begin{equation}
\mr{supp} \, \bs \chi^{\min} = \mr{epi} \,  \phi_1(\bs u) \cap \Big(\mathbb R^+ \times \mathbb R \times (w^{\min}, w^{\min}+r) \Big)
\end{equation}
where $\mr{hyp} \, \phi_1(\bs u)$ and $\mr{epi}\,  \phi_1(\bs u)$ denote the hypograph and epigraph, respectively, of the function $\phi_1(\bs u)$:
$$
\mr{hyp} \, \phi_1(\bs u) = \{(t, x, \xi) \; | \; \xi \leq \phi_1(\bs u(t,x))\}, \qquad \mr{epi} \, \phi_1(\bs u) = \{(t, x, \xi) \; | \; \xi \geq \phi_1(\bs u(t,x))\}
$$
With obvious notation, we also consider $\bs \psi^{\max}$, $\bs \psi^{\min}$. We have
\begin{equation}
 \partial_t \bs \chi^{\max} + \partial_x \bs \psi^{\max} = 0 \qquad \text{in $\msc D^\prime(\Omega \times  \mathbb R)$}
\end{equation}
\begin{equation}
     \partial_t \bs \chi^{\min} + \partial_x \bs \psi^{\min} = 0 \qquad \text{in $\msc D^\prime(\Omega \times  \mathbb R)$}.
\end{equation}
Notice that we can write 
$$
\bs \psi^{\max}(t,x,\xi) = \bs \lambda_1[\xi](\bs u(t,x)) \bs \chi^{\max}(t,x,\xi)
$$
$$
\bs \psi^{\min}(t,x,\xi) = \bs \lambda_1[\xi](\bs u(t,x)) \bs \chi^{\min}(t,x,\xi)
$$
where $\bs \lambda_1[\xi](\bs u)$ is as in \eqref{eq:lambdadef}; moreover, since the support of $\bs \chi^{\max}$ is contained in the strip $\mathbb R^2 \times [w_{\max}-r, w_{\max}]$, we deduce that for a.e. $(t,x, \xi)$

\begin{equation}\label{eq:chigsupp}
    \bs \chi^{\max}(t,x,\xi) \neq 0 \qquad \Longrightarrow \qquad w_{\max} - r < \xi < w_{\max}.
\end{equation}
 Therefore from Proposition \ref{prop:localspeed} we deduce that $\bs \chi^{\max}$ is uniformly positive in its support
\begin{equation}\label{eq:GNLk}
    \begin{aligned}
         \bs \chi^{\max}(t,x,\xi) \geq c  \cdot \mathbf 1_{ \mr{supp}\, \bs \chi^{\max}}(t,x,\xi) \qquad \text{for a.e. $t,x,\xi$}   
    \end{aligned}
\end{equation}
The same holds for $\bs \chi^{\min}$. 
Next, we want to apply  the Ambrosio superposition principle \cite[Theorem 3.2]{Amb08} to  the continuity equation in $\mathbb R_t \times \mathbb R^2$:
\begin{equation}\label{eq:conteq}
\partial_{t} \bs \chi^{\max} + \mr{div}_{x, \xi} \, \left(  (\bs \lambda_1[\xi](\bs u), 0)\cdot \bs \chi^{\max} \right) = 0.
\end{equation}
Our measure $\bs \chi^{\max}$ does not quite satisfy the assumption of  \cite[Theorem 3.2]{Amb08} since it is only locally finite, however our vector field is bounded. Therefore we will use the following version of the superposition principle, which follows with the same proof of \cite{Amb08}, or by a standard localization argument using finite speed of propagation. 
\begin{thm}\label{thm:supamb}
    Let $\{\mu_t\}_{t \in \mathbb R^+}\subset \msc M(\mathbb R^d)$ be a family of positive Radon measures satisfying 
    $$
    \partial_t \mu_t + \mr{div} (\bs b(t,x) \mu_t) = 0 \qquad \text{in $\msc D^\prime_{t,x}$}
    $$
    where $\bs b:\mathbb R^+\times \mathbb R^d \to \mathbb R^d$ is a Borel vector field satisfying $\|\bs b\|_{\mathbf L^{\infty}} < +\infty$. Then there is a measure $\bs \eta \in \msc M(\Gamma)$, concentrated on characteristic curves of $\bs b$, such that 
    $$
    \mu_t = (e_t)_\sharp \bs \eta
    $$
where $e_t(\gamma) = \gamma(t)$.
\end{thm}

Then we apply Theorem \ref{thm:supamb} to \eqref{eq:conteq}, and we obtain a positive measure $\bs \omega \in \msc M^+(\Gamma)$, where
$$
\Gamma = \Big\{\gamma = (\gamma_x, \gamma_\xi) : \mathbb R^+ \to \mathbb R^2, \qquad \gamma_x \quad \text{Lipschitz curve}, \qquad \gamma_\xi \in \mathbf L^\infty(\mathbb R^+) \Big\}
$$
such that 
\begin{enumerate}
    \item $\bs \omega$ is concentrated on curves $\gamma \in \Gamma$ such that
    \begin{enumerate}
    \item $\gamma_\xi$ is a constant function $\gamma_{\xi}(t) \equiv \xi_{\gamma} \in \mathbb R$ for all $t \in \mathbb R^+$;
        \item $\gamma_x$ is characteristic for $\bs \lambda_1[\xi](\bs u)$:
    \begin{equation}\label{eq:charomega}
    \dot \gamma_x(t) = \bs \lambda_1[\xi_{\gamma}](\bs u(t,x)) \qquad \text{for a.e. $t \in \mathbb R^+$}.
        \end{equation}
         \end{enumerate}
    \item Up to redefining $\bs \chi^{\max}$  on a set of times of measure zero, we can recover it by superposition of the curves:
    \begin{equation}\label{eq:chisuper}
\bs \chi^{\max}(t, \cdot, \cdot)\cdot \msc L^2 = (e_t)_\sharp \bs \omega  \qquad \text{for all $t \in \mathbb R^+$}.
        \end{equation}
    where $e_t : \Gamma \to \mathbb R^2$ is the evaluation map $e_t(\gamma) = \gamma(t)$.
\end{enumerate}

Entirely similar considerations hold for $\bs \chi^{\min}$; we thus call $\bs \eta$ the corresponding measure given by the Ambrosio superposition satisfying the same type of properties of (1), (2) above.

We now prove a preliminary lemma,  which is a version of \cite[Lemma 4]{Mar22} in the setting of Burgers equation.
\begin{lemma}\label{lemma:curves}
    For $\bs \omega$ almost every $\gamma = (\gamma_x, \xi_{\gamma}) \in\Gamma$, for $\msc L^1$-almost every $t \in \mathbb R^+$ it holds
    \begin{enumerate}
        \item $(t, \gamma_x(t))$ is a Lebesgue point of $\bs u$;
        \item it holds $w(t,\gamma_x(t))-r \leq \xi_{\gamma} \leq w(t,\gamma_x(t))$.
    \end{enumerate}
    Similarly, for $\bs \eta$ almost every $\sigma = (\sigma_x, \xi_{\sigma}) \in\Gamma$, for $\msc L^1$-almost every $t \in \mathbb R^+$ it holds
    \begin{enumerate}
        \item[(1')] $(t,\sigma_x(t))$ is a Lebesgue point of $\bs u$,
        \item[(2')] it holds $w(t,\sigma_x(t)) \leq \xi_{\sigma} \leq w(t,\sigma_x(t))+r$.
    \end{enumerate}
    We denote by $\Gamma^{\max}, \Gamma^{\min}$ the respective set of curves satisfying (1), (2) and (1'), (2'). 
\end{lemma}
\begin{proof}
    We prove the first half of the lemma, the second one being entirely symmetric. Let $S \subset \mathbb R^2$ be the set of non-Lebesgue points of $\bs u$; by the Lebesgue differentiation theorem we have $\msc L^2(S) = 0$. Denote by $e_t^x : \Gamma \to \mathbb R$ the map $e_t^x(\gamma) \doteq \gamma_x(t)$; then from \eqref{eq:chisuper} we deduce that for every $t \in \mathbb R$ it holds $(e_t^x)_\sharp\bs \omega \ll \msc L^1$, therefore
    $$
    \msc L^1 \otimes (e_t^x)_\sharp \bs \omega  \ll \msc L^2.
    $$
    Tonelli's theorem gives
    $$
    \begin{aligned}
        \int_{\Gamma} \msc L^1\left(\Big\{t \in \mathbb R \; | \; (t,\gamma_x(t)) \in S\Big\} \right) \dif \bs \omega(\gamma) & = \int_{\mathbb R} \bs \omega\left( \Big\{\gamma \in \Gamma \; | \; (t, \gamma_x(t)) \in S\Big\} \right) \dif t\\
        & = \left(\msc L^1 \otimes (e_t^x)_\sharp \bs \omega \right)(S) = 0.
    \end{aligned}
    $$
To prove (2), we first observe by Definition \ref{eq:chibcdef} that for every $t \in \mathbb R$, 
$$
\begin{aligned}
\bs \omega & \left(\Big\{ \gamma \in \Gamma \; | \; \xi_\gamma \notin \big(w(t, \gamma_x(t)) -r, w(t, \gamma_x(t)\big)\Big\} \right) \\
 =\,&  (e_t)_{\sharp} \bs \omega \left(\Big\{(x, \xi) \;  | \; \xi \notin \big(w(t,x)-r, w(t,x)\big) \Big\} \right) = 0
\end{aligned}
$$
and then we proceed as before using Tonelli's theorem to deduce
$$
\begin{aligned}
 & \int_{\Gamma} \msc L^1 \left(\Big\{t \in \mathbb R \; | \; \xi_\gamma \notin \big(w(t,\gamma_x(t))-r, w(t,\gamma_x(t))\big)\Big\} \right) \dif \bs \omega(\gamma) \\
 =\, &  \int_{\mathbb R} \bs \omega\left( \Big\{\gamma \in \Gamma \; | \; \xi_\gamma \notin \big(w(t,\gamma_x(t))-r, w(t,\gamma_x(t))\big) \Big\} \right) \dif t= 0.
        \end{aligned}
$$
\end{proof}

 For the proof of the following lemma we refer to \cite[Lemma 5]{Mar22}, in which the Lemma is proved for scalar functions, and the Lemma below follows by applying \cite[Lemma 5]{Mar22} twice on the components $\bs u = (u_1, u_2)$.
\begin{lemma}\label{lemma:traces}
Assume that $\gamma_x:(t_1,t_2)\subset \mathbb R\to \R$ is a Lipschitz curve and that for $\msc L^1$-a.e. $t \in (t_1,t_2)$ the point $(t, \gamma_x(t))$ is a Lebesgue point of $\bs u \in \mathbf L^\infty(\mathbb R^2; \,\mathbb R^2)$. Then
\begin{equation}\label{E_trace}
\lim_{\delta \to 0} \int_{t_1}^{t_2} \frac{1}{\delta}\int_{\gamma_x(t)-\delta}^{\gamma_x(t)+\delta}|\bs u(t,x)-\bs u(t,\gamma_x(t))|\dif x \dif t = 0.
\end{equation}
\end{lemma}

The following Lemma states that curves representing $\bs \chi^{\min}$ do not cross curves representing $\bs \chi^{\max}$, up to sets of measure zero.
\begin{lemma}\label{lemma:nocross}
    Let $\bar \sigma \in \Gamma^{\min}$, where $\Gamma^{\min}, \Gamma^{\max}$ are the sets of curves defined in Lemma \ref{lemma:curves}. Then 
    $$
\bs \omega \left( \Big\{ \gamma \in \Gamma^{\max} \; \big| \;   \exists \; 0 \leq t_1 < t_2 \; \text{with} \; (\gamma_x(t_1)-\bar \sigma_x(t_1)) \cdot  (\gamma_x(t_2)-\bar \sigma_x(t_2)) < 0 \Big\}\right) = 0.
$$
\end{lemma}
\begin{proof}
Let $\delta > 0$ be fixed, and $0  \leq t_1 < t_2$. Let 
$$
\phi^\delta(t,x) \doteq \begin{cases}
    0 & \text{if $x < \bar \sigma(t)-\delta$},\\
    \frac{x-\bar \sigma(t)+\delta}{\delta} & \text{if $\bar \sigma(t)-\delta \leq x \leq \bar \sigma(t)$},\\
    1 & \text{if $x \geq \bar \sigma(t) +\delta$}.
\end{cases}
$$
Consider 
$$
\Psi^\delta(t) \doteq  \int_{\Gamma} \phi^\delta(t,\gamma_x(t)) \dif \bs \omega(\gamma)
$$
and observe that $\Psi^\delta(t_1) = \mc O(1)\cdot \delta$, where $\mc O(1)$ is independent of $\delta$, while 
$$
\lim_{\delta \to 0^+} \Psi^\delta(t_2)  = \bs \omega(B_{t_1}^{t_2})
$$
where
$$
(B_{t_1}^{t_2})^\ell \doteq \Big\{\gamma \in \Gamma^{\max} \; \big| \;   \gamma_x(t_1) < \bar \sigma(t_1), \quad \gamma_x(t_2) > \bar \sigma (t_2) \Big\}.
$$
Then we have
$$
\begin{aligned}
    \Psi^\delta(t_2)&  = \Psi^\delta(t_1) + \int_{t_1}^{t_2}\int_{\Gamma^{\max}}\left(\, \frac{1}{\delta} \mathbf 1_{\{\gamma_x(t) \in (\bar \sigma_x(t) - \delta, \bar \sigma_x(t))\}}(\gamma)\cdot \left( (\dot \gamma_x(t)-\dot{\bar \sigma}_x(t)    \right) \, \right)\dif \bs \omega(\gamma) \\
    & \leq \mc O(1) \, \delta + \int_{\mathbb R} \int_{t_1}^{t_2} \frac{1}{\delta}\int_{\bar \sigma(t)-\delta}^{\bar \sigma(t)} \Big(\bs \chi^{\max}(t,x,\xi) \cdot \big(\dot \gamma_x(t) - \dot{\bar \sigma}_x(t) \big) \Big)\dif x \dif t \dif \xi\\
    & \leq \mc O(1) \delta + \mc O(1)\int_{t_1}^{t_2}\frac{1}{\delta}\int_{\bar \sigma(t)-\delta}^{\bar \sigma(t)} \mathbf 1_{\{w(t,x) > w_{\max} - r\}} \dif x \dif t .
\end{aligned}
$$
By Lemma \ref{lemma:curves} applied for $\bar \sigma \in \Gamma^{\min}$, we deduce that for almost every $t$, $(t,\bar \sigma(t))$ is a Lebesgue point of $w$ (because it is a.e. a Lebesgue point of $\bs u$ and $w = \phi_1(\bs u)$ where $\phi_1$ is a smooth function) and
$$
w(t,\bar \sigma(t)) \leq \xi_{\bar \sigma} < w_{\min} + r < w_{\max}-r \qquad \text{for a.e. $t$}. 
$$
Therefore, by Lemma \ref{lemma:traces} and Chebyshev's inequality we deduce that 
$$
\begin{aligned}
    & \int_{t_1}^{t_2} \frac{1}{\delta}\int_{\bar \sigma(t)-\delta}^{\bar \sigma(t)} \mathbf 1_{\{w(t,x) > w_{\max} - r\}} \dif x \dif t  \\
     \leq & \frac{1}{|w_{\max}-w_{\min}-2r|} \int_{t_1}^{t_2}\frac{1}{\delta}\int_{\bar \sigma(t)-\delta}^{\bar \sigma(t)}|w(t,x) - w(t, \bar \sigma(t))| \dif x \dif t \to 0 \quad \text{as $\delta \to 0^+$}.
\end{aligned}
$$
therefore we obtain 
$$
\bs \omega\left((B_{t_1}^{t_2})^\ell \right) = 0.
$$
A symmetric argument also shows $\bs \omega\left((B_{t_1}^{t_2})^r \right) = 0$, where
$$
(B_{t_1}^{t_2})^\ell \doteq \Big\{\gamma \in \Gamma^{\max} \; \big| \;   \gamma_x(t_1) > \bar \sigma(t_1), \quad \gamma_x(t_2) < \bar \sigma (t_2) \Big\}.
$$
By taking countable unions 
$$
B \doteq \bigcup_{\substack{q_1< q_2 \\ q_1, q_2 \in \mathbb Q^+}} (B_{q_1}^{q_2})^r \cup (B_{q_1}^{q_2})^\ell
$$
we see that since $\bs \omega$ is concentrated on curves such that $\gamma_x$ is Lipschitz, it holds
$$
B = \Big\{ \gamma \in \Gamma^{\max} \; \big| \;   \exists \; 0 \leq t_1 < t_2 \; \text{with} \; (\gamma_x(t_1)-\bar \sigma_x(t_1)) \cdot  (\gamma_x(t_2)-\bar \sigma_x(t_2)) < 0 \Big\}
$$
and this concludes the proof.
\end{proof}

\section{Liouville Type Theorem and VMO Points}\label{sec:LTVMO}
In this section we prove Theorem \ref{thm:Liuv}.

\begin{proof}[Proof of Theorem \ref{thm:Liuv}]
    To prove Theorem \ref{thm:Liuv}, we proceed by contradiction: assume that $\bs u$ is a non constant isentropic solution; we can assume that (say) the first Riemann invariant $w$ is non constant, with 
$$
w_{\min} < w_{\max}
$$
 where $w_{\min}, w_{\max}$ are as in \eqref{eq:maxind}, \eqref{eq:minind}.

 Then, by the previous section, without loss of generality, up to the change of time direction $t \mapsto -t$, there is an isentropic solution $\bs u : \mathbb R^+ \times \mathbb R$ such that there exist two curves $\bar \gamma, \bar \sigma$ in $\Gamma^{\max}, \Gamma^{\min}$  respectively (recall Lemma \ref{lemma:curves}), such that (see Figure \ref{fig:functional})
\begin{equation}\label{eq:contrhyp}
\bar \gamma_x(0)  < \bar \sigma_x(0) , \qquad b \doteq \xi_{\bar \gamma} > w_{\max}-r  \doteq a > w_{\min}+r >  \xi_{\bar\sigma} , \qquad \bar \gamma \in \Gamma^{\max}, \quad \bar \sigma \in \Gamma^{\min}
\end{equation}
with  $r > 0$  so that $r < \bar r$, where $\bar r $ is defined in  Proposition \ref{prop:localspeed} and
\begin{equation}\label{eq:contrhyp1}
\bar \gamma_x(t) \leq \bar \sigma_x(t) \qquad \forall \; t > 0.
\end{equation}
In fact, if the first condition of \eqref{eq:contrhyp} is not satisfied for $\bs u$, it is sufficient to consider the isentropic solution $\bs v(t,x) \doteq \bs u(-t, -x)$.

A contradiction will be reached by introducing a suitable interaction functional $\mc Q(t)$, constructed as follows. We define
\begin{equation}\label{eq:Qdef}
    \mc Q(t) \doteq \int_{a}^{b} \int_{\bar \gamma_x(t)}^{\bar \sigma_x(t)} \bs \chi^{\max}(t,x,\xi) \dif x \dif  \xi, \qquad t \geq 0.
\end{equation}
We now use the following Proposition (that we prove immediately after this proof), which is a key point of the paper and it shows that the functional $\mc Q$ is uniformly decreasing in time.
\begin{prop}\label{lemma:inter}
Assume that $\bs u:(0,+\infty) \times \mathbb R \to \mc U$ is an isentropic solution, and that there exist curves $\bar \gamma, \bar \sigma$ satisfying \eqref{eq:contrhyp}, \eqref{eq:contrhyp1}. Then if $\mc Q$ is as in \eqref{eq:Qdef}, there is $C>0$ such that for every $t > 0$ it holds
\begin{equation}
    \mc Q(t)  - \mc Q(0) \leq -t \, C 
\end{equation}
\end{prop}

Assuming the proposition, we thus have
$$
-\mc Q(0) \leq \mc Q(t)  - \mc Q(0) \leq -t \, C  \qquad \forall \; t > 0
$$
which leads to a contradiction letting $t \to +\infty$,  since $\mc Q(0)< \infty$. 
This proves Theorem \ref{thm:Liuv}.
\end{proof}
Now we prove Proposition \ref{lemma:inter}.

\begin{figure}
    \centering
    \includegraphics[width=0.9\linewidth]{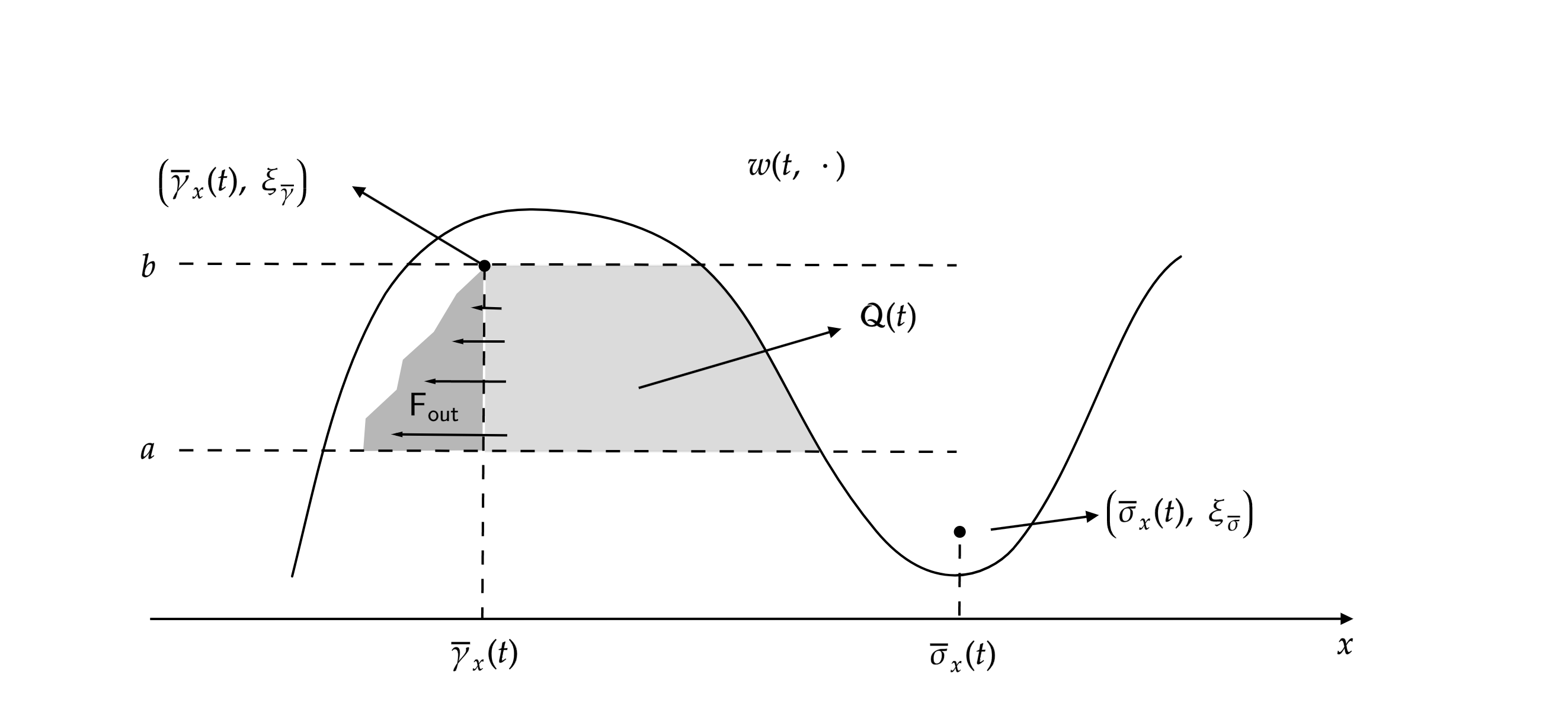}

    \caption{The gray area is proportional to the functional $\mc Q$. Due to genuine nonlinearity the rate of decrease $\mathsf F_{\mathsf{out}}$ of this area is bounded from below by a quantity independent on time.}
    \label{fig:functional}
\end{figure}

 We take a few lines to explain the heuristic behind the proof. Define 
$$
\rho(t,x) = \int_\mathbb R \bs \chi^{\max}(t,x,\xi) \mathbf 1_{(a,b)}(\xi) \dif \xi.
$$
We notice that 
$$
\mc Q(t) = \int_{\bar \gamma_x(t)}^{\bar \sigma_x(t)} \rho(t,x) \dif x
$$
and therefore that the variation of the functional $\mc Q(t)$ is related to the outward flux ${\mathsf{F}_{\mathsf {out}}}(t)$ of $\rho$ through the line $x = \bar \gamma_x(t)$ (i.e. the amount of mass of $\rho$ passing through $\bar \gamma_x(t)$ per unit time),  as well as to  its inward flux ${\mathsf{F}_{\mathsf {in}}}(t)$ through the line $x = \bar \sigma_x(t)$:
$$
\delta Q(t) = {\mathsf{F}_{\mathsf {in}}}(t) - {\mathsf{F}_{\mathsf {out}}}(t).
$$
By genuine nonlinearity (in particular by (2) of Proposition \ref{prop:localspeed}) the outward flux ${\mathsf{F}_{\mathsf {out}}}(t)$ through $\bar \gamma_x(t)$ is strictly positive, and bounded below independently of time (see Figure \ref{fig:functional}):
$$
{\mathsf{F}_{\mathsf {out}}}(t) \geq C > 0 \qquad \forall \, t > 0.
$$
In fact we will prove that 
$$
\mathsf F_{\mathsf{out}}(t) =  \int_a^b\Big(\bs \chi[\xi](\bs u(s,\bar \gamma_x(s)))\big(  \bs \lambda_1[b](\bs u(s, \bar \gamma_x(s)) -\bs \lambda_1[\xi](\bs u(s, \bar \gamma_x(s)))\, \big)  \Big) \dif \xi
$$
and the integrand is uniformly positive if the map $\xi \mapsto \bs \lambda_1[\xi](\bs u(s, \bar \gamma_x(s)))$ is strictly increasing (recall \eqref{eq:lambdadef}): this is the only point where the genuinely nonlinearity assumption comes into play.
On the other hand, since along the curve $\bar \sigma_x$ we have, by Lemma \ref{lemma:curves}, that $w(t, \bar \sigma_x(t))  <a$, we will deduce 
$$
{\mathsf{F}_{\mathsf {in}}}(t)  =  \int_a^b \bs \chi[\xi](\bs u(t, \bar \sigma(t))(\bs \lambda_1[b](\bs u(s, \bar \sigma_x(t))) - \bs \lambda_1[\xi](\bs u(t, \bar \gamma_x(t))) \dif x   = 0\qquad \forall \, t > 0.
$$
In turn this implies that the functional $\mc Q(t)$ is uniformly decreasing for all positive times, but since $\mc Q(0)$ is finite, this yields a contradiction.

\begin{proof}[Proof of Proposition \ref{lemma:inter}]
{\bf 1.} We consider appropriate regularizations of the interaction functional $\mc Q$ defined in the following way. Define first
$$
\varphi^{\delta}(t,x) \doteq \begin{cases}
    0, & \text{if $x \leq \bar \gamma_x(t) - \delta$}, \\
    \frac{x-\bar \gamma_x(t)+\delta }{\delta}, & \text{if $\bar \gamma_x(t) -\delta \leq x \leq \bar \gamma_x(t)$},\\
    1, & \text{if $\bar \gamma_x(t) \leq x \leq \bar \sigma_x(t)$}, \\
    1-\frac{x-\bar \sigma_x(t)}{\delta}, & \text{if $\bar \sigma_x(t) \leq x \leq \bar \sigma_x(t) + \delta$}, \\
    0, & \text{if $x \geq \bar \sigma_x(t) + \delta$}
\end{cases}
$$
and 
$
\phi^\delta(t,x,\xi) \doteq \varphi^\delta(t,x) \cdot \mathbf 1_{(a,b)}(\xi).
$
We define
\begin{equation}
    \mc Q^\delta(t) \doteq \int_a^b \int_{\mathbb R}  \phi^{\delta}(t, x,\xi) \, \bs \chi^{\max}(t,x,\xi) \dif x \dif  \xi 
\end{equation}
We observe that $\mc Q^\delta$ is Lipschitz continuous and we compute its derivative: notice that we can rewrite the functional as
$$
\mc Q^\delta(t) = \iint \phi^\delta(t, x, \xi) \dif (e_t)_\sharp \bs \omega(x,\xi) =   \int_{\Gamma} \phi^\delta(t, \gamma_x(t), \xi_\gamma) \dif \bs \omega(\gamma)
$$
therefore
$$
\lim_{h \to 0^+} \frac{\mc Q^\delta(t+h)-\mc Q^\delta(t)}{h}  = \lim_{h \to 0^+} \,\int_{\Gamma} \frac{\phi^{\delta}(t+h, \gamma_x(t+h),\xi_\gamma) - \phi^{\delta}(t, \gamma_x(t),\xi_\gamma)}{h} \dif \bs \omega(\gamma).
$$
Since there is $L> 0$ such that $\gamma_x$ is $L$-Lipschitz for $\bs \omega$-a.e. $\gamma \in \Gamma$, we have 
$$
\left|\frac{\phi^{\delta}(t+h, \gamma_x(t+h)) - \phi^{\delta}(t, \gamma_x(t))}{h}\right|  \leq L \,\sup |\nabla \phi^\delta| 
$$
Therefore, we conclude by the dominated convergence theorem that $\mc Q^\delta$ is  Lipschitz in time and that 
\begin{equation}
    \frac{\dif}{\dif t}{\mc Q}^\delta(t) = \int_{\Gamma} \left(\, \partial_t \phi^\delta(t, \gamma_x(t),\xi_\gamma) + \dot \gamma_x(t) \, \partial_x \phi^\delta(t, \gamma_x(t),\xi_\gamma) \,\right) \dif \bs \omega(\gamma).
\end{equation}
Using the definition of $\phi^\delta$, by \eqref{eq:chisuper}, we now rewrite
$$
\begin{aligned}
     \frac{\dif}{\dif t}{\mc Q}^\delta(t)  = &  -\int_{\Gamma} \left(\, \frac{1}{\delta} \mathbf 1_{\{\gamma(t) \in (\bar \gamma_x(t) - \delta, \bar \gamma_x(t)) \times (a,b)\}}(\gamma)\left( (  \dot{\bar \gamma}_x(t)-\dot \gamma_x(t) \right)\, \right) \dif \bs \omega(\gamma) \\
     & +   \int_{\Gamma}\left(\, \frac{1}{\delta}\mathbf 1_{\{\gamma(t) \in (\bar \sigma_x(t) , \bar \sigma_x(t)+\delta) \times (a,b)\}}(\gamma)\left( (\dot{\bar \gamma}_x(t)-\dot \gamma_x(t) \right) \, \right)\dif \bs \omega(\gamma) \\
      & \doteq -\ms {F^\delta_{out}}(t) + \ms{F^\delta_{in}}(t).
\end{aligned}
$$
Therefore we found
\begin{equation}
    \mc Q(t) = \mc Q(0) +  \lim_{\delta \to 0^+} \int_0^t -\ms {F^\delta_{out}}(s) +\ms{F^\delta_{in}}(s) \dif s.
\end{equation}

\vspace{0.3cm}
\noindent {\bf 2.} 
In this step we prove that for a constant $C$ independent of time, we have
\begin{equation}\label{eq:Fout}
    \lim_{\delta \to 0^+} \int_0^t \ms {F^\delta_{out}}(s)  \dif s\geq C \, t.
\end{equation}
In fact, using that 
$$
\dot{\bar \gamma}_x(s) =  \bs \lambda_1[b](\bs u(s, \bar \gamma_x(s))) \qquad \text{for a.e. $ s > 0$}
$$
we deduce
\begin{equation}\label{eq:Foutcomp}
\begin{aligned}
       \int_0^t   \ms {F^\delta_{out}}(s)  \dif s & =  \int_0^t \int_{\Gamma} \left(\, \frac{1}{\delta} \mathbf 1_{\{\gamma(s) \in (\bar \gamma_x(s) - \delta, \bar \gamma_x(s)) \times (a,b)\}}(\gamma)\left( (  \dot{\bar \gamma}_x(s)-\dot \gamma_x(s) \right)\, \right) \dif \bs \omega(\gamma) \dif s \\
     & =     \int_a^b \int_0^t \frac{1}{\delta}\int_{\bar \gamma_x(s) - \delta}^{\bar \gamma_x(s)} \Big( \bs \chi[\xi](\bs u(s,x))\big( \bs \lambda_1[b](\bs u(s, \bar \gamma_x(s))-\bs \lambda_1[\xi](\bs u(s,x)) \big) \Big) \dif x  \dif s \dif \xi.
\end{aligned}
\end{equation}
We claim that by Lemma \ref{lemma:traces}, we have that for every $\xi \in (a,b)$, it holds
\begin{equation}\label{eq:tr->flux}
\begin{aligned}
& \lim_{\delta \to 0^+}     \int_0^t \frac{1}{\delta}\int_{\bar \gamma_x(s) - \delta}^{\bar \gamma_x(s)} \Big( \bs \chi[\xi](\bs u(s,x))\big(\bs \lambda_1[b](\bs u(s, \bar \gamma_x(s))-\bs \lambda_1[\xi](\bs u(s,x)) \, \big) \Big) \dif x  \dif s\\
=  & \int_0^t \Big(\bs \chi[\xi](\bs u(s,\bar \gamma_x(s)))\big(  \bs \lambda_1[b](\bs u(s, \bar \gamma_x(s)) -\bs \lambda_1[\xi](\bs u(s, \bar \gamma_x(s)))\, \big)  \Big) \dif s.
\end{aligned}
\end{equation}
If $\bs\chi[\xi](\bs u)$ was Lipschitz, \eqref{eq:tr->flux} would follow easily from Lemma \ref{lemma:traces}, but $\bs \chi[\xi](\bs u)$ has a jump along the curve $\{\phi_1(\bs u) = \xi\}$ (recall that  $\phi_1$ is the first Riemann invariant \eqref{eq:riemdef}). Therefore we need to proceed in two steps: if $L>0$ is an upper bound for the Lipschitz constant of $\mc U \ni \bs u \mapsto \bs \lambda_1[\xi](\bs u)$ and of $\{\phi_1(\bs u) > \xi\} \ni \bs u \mapsto \bs \chi[\xi](\bs u)$,  first, using Lemma \ref{lemma:traces}, we estimate
\begin{equation}\label{eq:est1}
\begin{aligned}
& \int_0^t \frac{1}{\delta} \int_{\bar \gamma_x(s)  - \delta}^{\bar \gamma_x(s)} \Big| \bs \lambda_1[\xi](\bs u(s,x))- \bs \lambda_1[\xi](\bs u(s, \bar \gamma_x(s))) \Big| \dif x \dif s\\
\leq\,  & L \cdot\int_0^t \frac{1}{\delta}\int_{\bar \gamma_x(s) - \delta}^{\bar \gamma_x(s)} \big|\bs u(s,x)- \bs u(s, \bar \gamma_x(s))\big|\dif x \dif s \to 0 \quad \text{as $\delta \to 0^+$},
\end{aligned}
\end{equation}
\begin{equation}\label{eq:est2}
\begin{aligned}
    & \int_0^t \frac{1}{\delta} \int_{\bar \gamma_x(s)  - \delta}^{\bar \gamma_x(s)} \mathbf 1_{\{w(s,x) > \xi\}}(s,x)\Big| \bs \chi[\xi](\bs u(s,x))- \bs \chi[\xi](\bs u(s,\bar \gamma_x(s))) \Big|\dif x \dif s \\
     \leq\, & L \cdot \int_0^t \frac{1}{\delta}\int_{\bar \gamma_x(s) - \delta}^{\bar \gamma_x(s)}  \mathbf 1_{\{w(s,x) > \xi\}}(s,x)\big|\bs u(s,x)- \bs u(s, \bar \gamma_x(s))\big|\dif x \dif s \to 0 \quad \text{as $\delta \to 0^+$}.
    \end{aligned}
\end{equation}
Moreover, another application of Lemma \ref{lemma:traces} together with Chebyshev's inequality yields
\begin{equation}\label{eq:est3}
    \begin{aligned}
         & \int_0^t \frac{1}{\delta} \int_{\bar \gamma_x(s)  - \delta}^{\bar \gamma_x(s)} \mathbf 1_{\{w(s,x) < \xi\}}(s,x)\Big| \bs \chi[\xi](\bs u(s,x))- \bs \chi[\xi](\bs u(s,\bar \gamma_x(s))) \Big| \dif x \dif s \\
        \leq\, &  2 \sup_{\bs u} \bs \chi[\xi] \, \int_0^t \frac{1}{\delta} \int_{\bar \gamma_x(s)  - \delta}^{\bar \gamma_x(s)} \mathbf 1_{\{w(s,x) < \xi\}}(s,x) \dif x \dif s \\
        \leq\, &   2\sup_{\bs u} \bs \chi[\xi] \,  \,  \int_0^t \frac{1}{|w(s, \bar \gamma_x(s))-\xi|}\frac{1}{\delta} \int_{\bar \gamma_x(s)  - \delta}^{\bar \gamma_x(s)}|w(s,x) - w(s, \bar \gamma_x(s))| \dif x \dif s \\
             \leq \,&   2\sup_{\bs u} \bs \chi[\xi] \, \frac{1}{|\xi_{\bar \gamma}-\xi|} \,  \int_0^t \frac{1}{\delta} \int_{\bar \gamma_x(s)  - \delta}^{\bar \gamma_x(s)}\big|w(s,x) - w(s, \bar \gamma_x(s))\big| \dif x \dif s  \to 0 \quad \text{as $\delta \to 0^+$}
    \end{aligned}
\end{equation}
where in the second inequality we used the Chebyshev's inequality (recall $\xi < \xi_{\bar \gamma} =b<w(s, \bar \gamma_x(s))$): 
$$
 \int_{\bar \gamma_x(s)  - \delta}^{\bar \gamma_x(s)} \mathbf 1_{\{w(s, \bar \gamma_x(s)) -w(s,x) >w(s, \bar \gamma_x(s))-\xi\}}(s,x)  \dif s \leq \frac{1}{|w(s, \bar \gamma_x(s))-\xi|} \int_{\bar \gamma_x(s)  - \delta}^{\bar \gamma_x(s)}|w(s,x)-\xi_{\bar \gamma}|  \dif s
$$
and in the last inequality the fact that $w(s, \bar \gamma_x(s)) > \xi_{\bar \gamma}$.
Summing \eqref{eq:est2}, \eqref{eq:est3} we deduce that 
\begin{equation}\label{eq:est4}
    \int_0^t \frac{1}{\delta} \int_{\bar \gamma_x(s)  - \delta}^{\bar \gamma_x(s)} \Big| \bs \chi[\xi](\bs u(s,x))- \bs \chi[\xi](\bs u(s,\bar \gamma_x(s))) \Big|\dif x \dif s\to 0 \quad \text{as $\delta \to 0^+$}. 
\end{equation}
Finally, \eqref{eq:tr->flux} follows just by lengthy but trivial triangular inequalities:
\begin{equation}\label{eq:tr->fluxproof}
\begin{aligned}
& \Bigg| \int_0^t \frac{1}{\delta}\int_{\bar \gamma_x(s) - \delta}^{\bar \gamma_x(s)} \Big( \bs \chi[\xi](\bs u(s,x))\big(\bs \lambda_1[b](\bs u(s, \bar \gamma_x(s))-\bs \lambda_1[\xi](\bs u(s,x)) \, \big) \Big) \\
& - \Big(\bs \chi[\xi](\bs u(s,\bar \gamma_x(s)))\big(  \bs \lambda_1[b](\bs u(s, \bar \gamma_x(s)) -\bs \lambda_1[\xi](\bs u(s, \bar \gamma_x(s)))\, \big)  \Big) \dif x  \dif s\Bigg|\dif x \dif s   \\
 \leq\, &   \sup \bs \lambda_1 \cdot\int_0^t \frac{1}{\delta} \int_{\bar \gamma_x(s) - \delta}^{\bar \gamma_x(s)}\Big| \bs \chi[\xi](\bs u(s,x)) - \bs \chi[\xi](\bs u(s,\bar \gamma_x(s)))\Big|
  \\
 +\, & \sup \bs \chi[\xi]  \int_0^t \frac{1}{\delta}\int_{\bar \gamma_x(s) - \delta}^{\bar \gamma_x(s)} \Big|  \bs \lambda_1[\xi](\bs u(s, x) -\bs \lambda_1[\xi](\bs u(s, \bar \gamma_x(s)))\,  \Big| \dif x \dif s\to 0 \quad \text{as $\delta \to 0^+$}
\end{aligned}
\end{equation}
where all the terms in the last two lines the limit as $\delta \to 0^+$ is zero thanks to \eqref{eq:est1}, \eqref{eq:est4}. This proves \eqref{eq:tr->flux}. 

Next, using \eqref{eq:chigsupp}, \eqref{eq:GNLk}, and Proposition \ref{prop:localspeed}, we estimate the right hand side in \eqref{eq:tr->flux} by
\begin{equation}\label{eq:usoGNLk}
\int_0^t \Big(\bs \chi[\xi](\bs u(s,\bar \gamma_x(s)))\big(  \bs \lambda_1[b](\bs u(s, \bar \gamma_x(s)) -\bs \lambda_1[\xi](\bs u(s, \bar \gamma_x(s)))\, \big)  \Big) \dif s\geq 
t c^2  (b-\xi).
\end{equation}
Integrating also in $\xi$ we finally obtain, combining \eqref{eq:Foutcomp}, \eqref{eq:usoGNLk}, and the dominated convergence theorem,
$$
\lim_{\delta \to 0^+}  \int_0^t   \ms {F^\delta_{out}}(s) \dif s  \geq t c^2 \int_a^b (b-\xi)  \geq  t\cdot c^2 \frac{(b-a)^2}{2}
$$
so that \eqref{eq:Fout} is proved.

\vspace{0.3cm}
\noindent {\bf 3.} This steps concludes the proof by showing that
\begin{equation}\label{eq:Fin}
    \lim_{\delta \to 0^+} \int_0^t \ms {F^\delta_{in}}(s)  \dif s = 0.
\end{equation}
We have, as above, 
$$
\begin{aligned}
    \int_0^t  \ms {F^\delta_{in}}(s)  \dif s  
    &= \int_{\Gamma}\left(\, \frac{1}{\delta} \mathbf 1_{\{\gamma(s) \in (\bar \sigma_x(s) , \bar \sigma_x(s)+\delta) \times (a,b)\}}(\gamma)\cdot \left( (\dot{\bar \gamma}_x(s)-\dot \gamma_x(s)    \right) \, \right)\dif \bs \omega(\gamma) \\
    & =  \int_a^b \int_0^t \frac{1}{\delta}\int_{\bar \sigma_x(s)}^{\bar \sigma_x(s)+\delta} \Big( \bs \chi[\xi](\bs u(s,x))\big(\bs \lambda_1[b](\bs u(s, \bar \sigma_x(s))-\bs \lambda_1[\xi](\bs u(s,x))  \big) \Big) \dif x  \dif s \dif \xi\\
    & \leq \max \big( 2|\bs \chi| |\bs \lambda_1|\big) \cdot \int_a^b \int_0^t \frac{1}{\delta}\int_{\bar \sigma_x(s)}^{\bar \sigma_x(s)+\delta} \mathbf 1_{\{w(t,x) \geq b\}}(s,x) \dif x  \dif s \dif \xi \\
    & \leq (b-a) \cdot \max \big( 2|\bs \chi| |\bs \lambda_1|\big)\cdot \int_0^t \frac{1}{\delta}\int_{\bar \sigma_x(s)}^{\bar \sigma_x(s)+\delta} \mathbf 1_{\{w(t,x) \geq b\}}(s,x) \dif x \dif s
\end{aligned} 
$$
By Lemma \ref{lemma:curves} applied for $\bar \sigma \in \Sigma$, we deduce that $(s,\bar \sigma(s))$ is a Lebesgue point of $w$ (because it is a Lebesgue point of $U$ and  $w =\phi_1(U)$ where $\phi_1$ is a smooth function) and
$$
w(s,\bar \sigma(s)) \leq \xi_{\bar \sigma} < a \qquad \text{for a.e. $s$}. 
$$
Therefore, by Lemma \ref{lemma:traces} and Chebyshev's inequality we deduce that 
$$
\begin{aligned}
    & \int_0^t \frac{1}{\delta} \int_{\bar \sigma_x(s)}^{\bar \sigma_x(s)+\delta} \mathbf 1_{\{w(t,x) \geq b\}}(s,x) \dif x \dif s   \\
    \leq\, & \frac{1}{|a-\xi_{\bar \sigma}|}\int_0^t \frac{1}{\delta}\int_{\bar \sigma_x(s)}^{\bar \sigma_x(s)+\delta}  |w(s,x)-w(s,\bar \sigma(t))| \dif x \dif s \to 0 \quad \text{as $\delta \to 0^+$}.
\end{aligned}
$$
This proves \eqref{eq:Fin}, and ultimately it proves the proposition.
\end{proof}

\subsection{VMO regularity outside of J}
\begin{thm}\label{thm:VMO}
    Let $\bs u:\mathbb R^+ \times \mathbb R \to \mc U$ be a finite entropy solution to \eqref{eq:system} and let $\br J$ be the set in \eqref{eq:Jdef}. Assume that the eigenvalues are genuinely nonlinear. Then every point $(\bar t,\bar x) \in (0, +\infty) \times \mathbb  R \setminus \br J$ is of vanishing mean oscillation, i.e.
    $$
    \lim_{r \to 0^+} \frac{1}{r^2} \int_{B_{r}((\bar t,\bar x))} \left|\bs u(y)-\Big(\fint_{B_r((\bar t,\bar x))} \bs u \Big)\right| \dif y = 0 \qquad \forall \; (\bar t,\bar x) \in (0, +\infty) \times \mathbb R \setminus \br J.
    $$
\end{thm}
\begin{proof}
    Let $(\bar t,\bar x) \in (0, +\infty) \times \mathbb  R \setminus \br J$. Define $\bs u_r : \mathbb R^2 \to \mc U$ by
    $$
\bs u_r(t,x) \doteq \begin{cases}
    \bs u(\bar t + r(t-\bar t), \bar x + r(x-\bar x)) & \text{if $t > \bar t-\frac{1}{r}\bar t$}, \\
    0 & \text{otherwise}.
\end{cases}
$$
Assume by contradiction that at some point $(\bar t, \bar x) \in \mathbb R \setminus \br J$, i.e. that 
\begin{equation}\label{eq:Jdef1}
\limsup_{r \to 0^+} \frac{\bs \nu(B_r(\bar t,\bar x)) }{r} = 0.
\end{equation}
and also that for some subsequence $\{r_j\}_j$ it holds
\begin{equation}\label{eq:notvmo}
\lim_{j \to +\infty} \fint_{B_{r_j}(\bar t, \bar x)} |\bs u(t,x) - \overline {\bs u}_{r_j}(\bar t, \bar x)| \dif x \dif t   > 0
\end{equation}
where $\overline {\bs u}_r(\bar t, \bar x)$ is defined by 
$$
\overline {\bs u}_r(\bar t, \bar x) \doteq \fint_{B_r(\bar t, \bar x)} \bs u(t,x) \dif x \dif t.
$$
Up to a further subsequence, we can also assume that 
\begin{equation}\label{eq:barurconv}
\overline {\bs u}_{r_j}(\bar t, \bar x) \longrightarrow \overline {\bs u} \in \mc U
\end{equation}
\begin{equation}\label{eq:urstrong}
\bs u_{r_j} \longrightarrow  \bs v \quad \text{strongly in $\mathbf L^1_{loc}(\mathbb R^2)$}.
\end{equation}
Indeed, the sequence $\bs u_{r_j}$ is strongly compact in $\mathbf L^1_{loc}$ thanks to compensated compactness. For this well known fact we refer to   \cite{Ser00}, but for convenience of the reader we write more precisely how to combine the various statements in \cite{Ser00}. In particular by \cite[Chapter 9, Proposition 9.1.5]{Ser00}, up to a further sequence, the limit of $\bs u_{r_j}$ exists in the sense of Young measures, i.e. there is a measurable map $(t,x) \mapsto \alpha_{t,x} \in \mathscr P(\mc U)$ (the set of probability measures on $\mc U$) such that for every smooth function $\psi : \mc U \to \mathbb R$ there holds
$$
\psi(\bs u_{r_j}) \rightharpoonup^\ast \int \psi(\bs q) \dif \alpha_{t,x}(\bs q) \qquad \text{weakly$^\ast$ in $\mathbf L^{\infty}$ as $j \to +\infty$.}
$$
Then, by \cite[Chapter 9, Proposition 9.1.7]{Ser00} the sequence $\bs u_{r_j}$ converges strongly in $\mathbf L^1_{loc}$ if and only if $\alpha_{t,x}$ has support in a single point for a.e. $(t, x) \in \mathbb R^2$. But this follows from \cite[Chapter 9, Proposition 9.2.2 and Proposition 9.51]{Ser00}.

Next, we show that $\bs v$ is a global isentropic solution.
Let $\eta, q$ be a smooth entropy-entropy flux pair and $\varphi \in C^1_c(\mathbb R^2)$, and consider $R > 0$ so that $\mr{supp}\, \varphi \subset B_R \subset \mathbb R^2$. We compute, using \eqref{eq:Jdef1} in the last line,
$$
\begin{aligned}
   \left| \iint_{\mathbb R^2} \varphi_t\eta(\bs v) + \varphi_xq(\bs v) \dif x \dif t\right| &  = \lim_{j \to +\infty}  \left|\iint_{\mathbb R^2} \varphi_t \eta(\bs u_{r_j}) + \varphi_x q(\bs u_{r_j} )\dif x \dif t\right|\\
    & = \lim_{j \to +\infty}\frac{1}{r_j} \left|\iint_{\mathbb R^2} \widetilde \varphi_t \eta (\bs u) + \widetilde \varphi_x q(\bs u) \dif x\dif t\right| \\
    &  = \lim_{j \to +\infty} \frac{1}{r_j} \left|\iint_{\mathbb R^2} \widetilde \varphi \dif \mu_\eta(t,x)\right|\\
    & \leq  \lim_{j \to +\infty} \frac{1}{r_j} \iint_{\mathbb R^2} \left|\widetilde \varphi\right| \dif \bs \nu (t,x) \\
    & \leq \|\varphi\|_{C^0}  \limsup_{j \to +\infty}\frac{\bs \nu(B_{r_j}(\bar t, \bar x)}{r_j} \longrightarrow 0 \quad \text{ as $ j \to +\infty$}
\end{aligned}
$$
where here
$$
\widetilde \varphi(t,x) \doteq  \varphi\left(\bar t+ \frac{t-\bar t}{r_j}, \, \bar x+ \frac{x-\bar x}{r_j} \right), \qquad \mr{supp} \, \wt \varphi \subset B_{r_j}(\bar t, \bar x)
$$
and $\mu_\eta$ is as in \eqref{eq:muetadef}.

Applying Theorem \ref{thm:Liuv} we deduce that $\bs v$ must be a constant: $\bs v(t,x) \equiv \overline {\bs v}$ for a.e. $(t, x) \in \mathbb R^2$, for some $\overline{\bs v} \in \mc U$. Now notice that $\bs v \equiv \overline {\bs u}$, because 
$$
\overline{\bs u} = \lim_j \fint_{B_{r_j}(\bar t, \bar x)} \bs u(t,x) \dif x \dif t  = \lim_j\fint_{B_1(0)} \bs u_{r_j}(s, y) \dif y \dif s = \overline {\bs v}
$$
so that 
\begin{equation}\label{eq:vconst}
    \overline{\bs u} = \overline{ \bs v} = \bs v(t,x) \qquad \text{for a.e.} \; (t,x) \in \mathbb R^2.
\end{equation}
But then we have a contradiction because 
$$
\begin{aligned}
    0 = \lim_j \fint_{B_1(0)} |\bs u_{r_j}(t,x)- \overline{\bs v}| \dif x \dif t & = \lim_j \fint_{B_1(0)} |\bs u_{r_j}(t,x)- \overline{\bs u}| \dif x \dif t \\
    & = \lim_j \fint_{B_{r_j}(\bar t, \bar x)} |\bs u(t,x)- \overline{\bs u}| \dif x \dif t > 0
\end{aligned}
$$
where we used \eqref{eq:urstrong}, \eqref{eq:vconst} in the first and second equality, and \eqref{eq:notvmo} in the last inequality.
\end{proof}

\noindent {\bf Acknowledgements.}  
F.~Ancona, E.~Marconi and L.~Talamini are members of GNAMPA of the ``Istituto Nazionale di Alta Matematica
F.~Severi". F.~Ancona and E.~Marconi are partially supported by the PRIN
Project 20204NT8W4 “Nonlinear evolution PDEs, fluid dynamics and transport
equations: theoretical foundations and applications”,
and by the PRIN 2022 PNRR Project P2022XJ9SX ``Heterogeneity on the road - modeling, analysis, control''.
E.~Marconi  is also partially supported by H2020-MSCA-IF “A~Lagrangian approach: from conservation laws to line-energy Ginzburg-Landau models”. Luca Talamini is partially supported by the GNAMPA - Indam Project 2025 \say{Rappresentazione lagrangiana per sistemi di leggi di conservazione ed equazioni cinetiche.}


\vspace{1cm}



\begin{thebibliography}{DLO03a}


\bibitem[AFP00]{AFP00}  
Ambrosio, L., Fusco, N., and Pallara, D.  
\newblock {\em Functions of Bounded Variation and Free Discontinuity Problems}.  
\newblock Oxford Science Publications. Clarendon Press, 2000.



\bibitem[Amb08]{Amb08}  
Ambrosio, L.,  
\emph{Transport equation and Cauchy problem for non-smooth vector fields},  
Lecture Notes in Mathematics, vol. 1927, pp. 2–41, 2008.



\bibitem[ABB23]{ABB23} Ancona, F., Bianchini, S., Bressan, A., Colombo, R. M., Nguyen, K. T., Examples and conjectures on the regularity of solutions to balance laws, {\it Quarterly Appl. Math.}, {\bf 81}, 433--454, 2023.

\bibitem[ABM25]{ABM25} Ancona, F., Bressan, A., Marconi, E., Talamini, L., Intermediate domains for scalar conservation laws, \textit{Journal of Differential Equations}, \textbf{422}, 215--250, 2025.



\bibitem[AC05]{AC05}
Ancona, F., and Coclite, G. M., 
On the attainable set for Temple class systems with boundary controls, 
\textit{SIAM Journal on Control and Optimization}, vol. 43, no. 6, pp. 2166--2190, 2005.

\bibitem[AMT25]{AMT25}
Ancona, F., Marconi, E., and Talamini, L., 
\textit{Lagrangian representation for scalar balance laws and applications to systems of conservation laws}, 
in preparation, 2025.

 \bibitem[AT24a]{AT24a} Ancona, F., Talamini, L. \textit{Backward-Forward Characterization of Attainable Set for Conservation Laws with Spatially Discontinuous Flux}, arXiv:2404.00116, 2024.

 \bibitem[AT24b]{AT24b} Ancona, F., Talamini, L. \textit{Initial Data Identification for Conservation Laws with Spatially Discontinuous Flux}, arXiv:2408.00472, 2024.

\bibitem[BB05]{BB05} Bianchini, S., Bressan, A., Vanishing viscosity solutions of nonlinear hyperbolic systems, \textit{Annals of Mathematics}, \textbf{161}, no. 1, 223--342, 2005.

\bibitem[BBM17]{BBM17} Bianchini, S., Bonicatto, P., Marconi, E., A lagrangian approach to multidimensional conservation laws, \textit{preprint SISSA 36/MATE}, 2017.

\bibitem[BM17]{BM17}  
Bianchini, S. and Marconi, E., On the structure of -entropy solutions to scalar conservation laws in one-space dimension, \emph{Archive for Rational Mechanics and Analysis}, 226(1):441–493, 2017.


\bibitem[Bre00]{Bre00} Bressan, A., \textit{Hyperbolic systems of conservation
    laws.  The one-dimensional Cauchy problem}. Oxford University
  Press, Oxford, 2000.



\bibitem[BCZ18]{BCZ18}  
Bressan, A., Chen, G., and Zhang, Q.,  
\newblock On finite time BV blow-up for the p-system,  
\newblock \emph{Communications in Partial Differential Equations}, 43(8):1242--1280, 2018.



\bibitem[BDL23]{BDL23}
Bressan, A., and De Lellis, C., 
A remark on the uniqueness of solutions to hyperbolic conservation laws, 
\textit{Archive for Rational Mechanics and Analysis}, vol. 247, p. 106, 2023.


\bibitem[BC98]{BC98} Brenier, Y., Corrias, L., A kinetic formulation for multi-branch entropy solutions of scalar conservation laws, \textit{Ann. Inst. H. Poincaré Anal. Non Linéaire}, \textbf{15} (2), 1998.


\bibitem[CP10]{CP10} 
G.-Q. G. Chen and M. Perepelitsa, 
Vanishing viscosity limit of the Navier-Stokes equations to the Euler equations for compressible fluid flow,
\textit{Communications on Pure and Applied Mathematics}, vol. 63, no. 11, pp. 1469--1504, 2010.



\bibitem[CT11]{CT11}
Chen, G. Q., Torres, M. On the structure of solutions of nonlinear hyperbolic systems of conservation laws, \textit{Communications on Pure and Applied Analysis}, vol. 10, no. 4, pp. 1011–1036, 2011.


\bibitem[CKV22]{CKV22}
Chen, G. Q., Krupa, S.G., and Vasseur, A.F.,  
Uniqueness and Weak-BV Stability for Conservation Laws,  
Arch. Ration. Mech. Anal., 246, 299–332, 2022.

\bibitem[CVY24]{CVY24}
 Chen, R. M.,  Vasseur,A. F., and  Yu, C.,
\emph{Non-uniqueness for continuous solutions to 1D hyperbolic systems}, 
arXiv:2407.02927, 2024.


\bibitem[Daf16]{Daf16}  
Dafermos, C.,  
\emph{Hyperbolic Conservation Laws in Continuum Physics,}  
Fourth edition. Springer-Verlag, Berlin, 2016.

\bibitem[DOW03]{DOW03}  
De Lellis, C., Otto, F., and Westdickenberg, M.  
\newblock Structure of Entropy Solutions for Multi-Dimensional Scalar Conservation Laws.  
\newblock {\em Archive for Rational Mechanics and Analysis}, 170 (2), 2003.


\bibitem[DLR03]{DLR03}
De Lellis C., Riviere T.
\newblock The rectifiability of entropy measures in one space dimension
\newblock {\em Journal des Mathematiques Pures et Appliquees}, 82 (10), 1343-1367, 2003

\bibitem[DLM91]{DLM91}
Diperna, R. J., Lions, P. L., and Meyer, Y.,  
$L^p$ regularity of velocity averages,  
Ann. I.H.P. Anal. Non Linéaire, 8(3-4), 271–287, 1991.



\bibitem[DiP83a]{DiP83a} 
DiPerna, R. J.  
\newblock Convergence of approximate solutions to conservation laws. 
\newblock {\em Archive for Rational Mechanics and Analysis}, 82 (1), 1983.

\bibitem[DiP83b]{DiP83b}
DiPerna, R. J., 
Convergence of the viscosity method for isentropic gas dynamics, 
\textit{Communications in Mathematical Physics}, vol. 91, pp. 1--30, 1983.

\bibitem[FJ99]{FJ99}  
Friedlander, F. G., and Joshi, M.,  
\textit{Introduction to the Theory of Distributions}, 2nd ed., Cambridge University Press, 1999.

\bibitem[Gla24]{Gla24}  
Glass, O.,  
\newblock $2 \times 2$ hyperbolic systems of conservation laws in classes of functions of bounded p-variation,  
\newblock \emph{arXiv:2405.02123}, 2024.

\bibitem[GL70]{GL70}  
Glimm, J., and Lax, P.,  
\newblock Decay of solutions of systems of nonlinear hyperbolic conservation laws.  
\newblock {\it Memoirs of the American Mathematical Society}  101, Providence, R.I., 1970.


\bibitem[Gol23]{Gol23}
Golding, W., 
Unconditional regularity and trace results for the isentropic Euler equations with \(\gamma = 3\), 
\textit{SIAM Journal on Mathematical Analysis}, vol. 55, no. 5, pp. 5751--5781, 2023

\bibitem[Kru77]{Kru77} Kruzhkov, S., First-order quasilinear equations with
  several space variables, {\it Mat. Sb.} {\bf 123},
  228--255. English transl. in {\it Math. USSR Sb.} {\bf 10},
  217--273, 1970.


\bibitem[LPS96]{LPS96}
Lions, P.-L., Perthame, B., and Souganidis, P. E., 
Existence and stability of entropy solutions for the hyperbolic systems of isentropic gas dynamics in Eulerian and Lagrangian coordinates, 
\textit{Communications on Pure and Applied Mathematics}, vol. 49, pp. 599--638, 1996.




\bibitem[LPT94a]{LPT94a}
 Lions, P.-L., Perthame, B., and Tadmor, E.,
\newblock A kinetic formulation of multidimensional scalar conservation laws
  and related equations.
\newblock {\em J. Amer. Math. Soc.}, 7(1):169--191, 1994.

\bibitem[LPT94b]{LPT94b}  
Lions, P.-L., Perthame, B., and Tadmor, E.  
\newblock Kinetic formulation of the isentropic gas dynamics and p-systems.  
\newblock {\em Communications in Mathematical Physics}, 163(2), 1994.


\bibitem[Mar19]{Mar19}
Marconi, E.,
\newblock On the structure of weak solutions to scalar conservation laws with finite entropy production.
\newblock {\em Calculus of Variations and Partial Differential Equations}, 61(1), 2019.


\bibitem[Mar22]{Mar22}  
Marconi, E.,  
\newblock The rectifiability of the entropy defect measure for Burgers equation,  
\newblock \textit{Journal of Functional Analysis}, 283, no. 6, 2022.


\bibitem[Mar21]{Mar21}
Marconi, E., Characterization of Minimizers of Aviles–Giga Functionals in Special Domains, \textit{Archive for Rational Mechanics and Analysis}, vol. 242, pp. 1289–1316, 2021.


\bibitem[PT00]{PT00}  
Perthame, B., Tzavaras, A.,  
\newblock A Kinetic Formulation for Systems of Two Conservation Laws and Elastodynamics.  
\newblock {\it Archive for Rational Mechanics and Analysis}  155, 1–48, 2000.

\bibitem[Ser00]{Ser00}  
Serre, D.,  
\newblock Systems of Conservation Laws 2: Geometric Structures, Oscillations, and Initial-Boundary Value Problems.  
\newblock {\it Cambridge University Press}, 2000.

\bibitem[Sil18]{Sil18}  
Silvestre, L.,  
\newblock Oscillation Properties of Scalar Conservation Laws,  
\newblock {\em Comm. Pure. Appl. Math}, 72(6), 2018.


\bibitem[Tal24]{Tal24}
Talamini, L., 
\textit{Regularity, Decay, Differentiability for Solutions to Conservation Laws and Structural Properties for Conservation Laws with Discontinuous Flux}, 
Ph.D. Thesis, University of Padova, 2024


\bibitem[Tza03]{Tza03}  
Tzavaras, A.,  
\newblock The Riemann Function, Singular Entropies, and the Structure of Oscillations in Systems of Two Conservation Laws.  
\newblock {\it Archive for Rational Mechanics and Analysis}  169, 181--193, 2003.

\bibitem[Tar79]{Tar79}  
Tartar, L.,  
\newblock Compensated Compactness and Applications to Partial Differential Equations,  
\newblock \emph{Lecture Notes in Mathematics}, vol. 1310, Springer-Verlag, 1979.

\bibitem[Vas99]{Vas99}
Vasseur, A., 
Time regularity for the system of isentropic gas dynamics with \(\gamma = 3\), 
\textit{Communications in Partial Differential Equations}, vol. 24, no. 11–12, 1999.


\end{thebibliography}
\end{document}